\pdfoutput=1

\documentclass[10pt]{amsart}

\usepackage{amsbsy}
\usepackage{amscd}
\usepackage{amsfonts}
\usepackage{amsgen}
\usepackage{amsmath}
\usepackage{amsopn}
\usepackage{amssymb}
\usepackage{amstext}
\usepackage{amsthm}
\usepackage{amsxtra}
\usepackage[all]{xy}
\usepackage{comment}
\usepackage{euscript}
\usepackage{mathrsfs}
\usepackage{MnSymbol}
\usepackage{rotating}
\usepackage{url}

 \usepackage{hyperref}

\theoremstyle{plain}
\newtheorem{thm}{Theorem}[section]
\newtheorem{prop}[thm]{Proposition}
\newtheorem{lem}[thm]{Lemma}
\newtheorem{cor}[thm]{Corollary}

\theoremstyle{definition}\newtheorem{defn}[thm]{Definition}
\newtheorem{rmk}[thm]{Remark}
\newtheorem{example}[thm]{Example}

\newtheorem{note}[thm]{Notation}

\numberwithin{equation}{section}

\renewcommand{\theta}{\vartheta}
\renewcommand{\phi}{\varphi}
\renewcommand{\epsilon}{\varepsilon}
\renewcommand{\subset}{\subseteq}
\renewcommand{\supset}{\supseteq}

\newcommand{\N}{\mathbb N}

\newcommand{\C}{\mathbb C}

\newcommand{\ms}[1]{\mathscr{#1}}
\newcommand{\mc}[1]{\mathcal{#1}}

\newcommand{\bfi}{\mathbf i}
\newcommand{\bfj}{\mathbf j}
\newcommand{\bfr}{\mathbf r}
\newcommand{\bfs}{\mathbf s}

\newcommand{\hra}{\hookrightarrow}
\newcommand{\ra}{\rightarrow}

\newcommand{\tl}{\mathrm{nut}}

\newcommand{\singleton}{\uparrow}

\usepackage{calc}
\newcounter{PartitionDepth}
\newcounter{PartitionLength}

\newcommand{\upparti}[2]{
 \begin{picture}(#2,#1)
 \setcounter{PartitionDepth}{#1}
 \put(#2,0){\line(0,1){#1}}
 \end{picture}}
\newcommand{\uppartii}[3]{
 \begin{picture}(#3,#1)
 \setcounter{PartitionLength}{#3-#2}
 \setcounter{PartitionDepth}{#1}
 \put(#2,0){\line(0,1){#1}}
 \put(#3,0){\line(0,1){#1}}
 \put(#2,\thePartitionDepth){\line(1,0){\thePartitionLength}}
 \end{picture}}
\newcommand{\uppartiii}[4]{
 \begin{picture}(#4,#1)
 \setcounter{PartitionLength}{#4-#2}
 \setcounter{PartitionDepth}{#1}
 \put(#2,0){\line(0,1){#1}}
 \put(#3,0){\line(0,1){#1}}
 \put(#4,0){\line(0,1){#1}}
 \put(#2,\thePartitionDepth){\line(1,0){\thePartitionLength}}
 \end{picture}}

\begin{document}

\title[De Finetti theorems for Boolean EQG]{De Finetti theorems for a Boolean analogue of easy quantum groups}
\author{Tomohiro Hayase}
\address{Graduate School of Mathematics\\University of Tokyo\\Komaba, Tokyo 153-8914, Japan}
\email{\href{mailto:}{hayase@ms.u-tokyo.ac.jp}}
\date{\today}

\begin{abstract}
We show an organized form of quantum de Finetti theorem for Boolean independence.  We define  a Boolean analogue of easy quantum groups for the categories of interval partitions, which is a family of sequences of quantum semigroups.

We construct the Haar states on those quantum semigroups. The proof of our de Finetti theorem is based on the analysis of the Haar states.   
\end{abstract}

\maketitle

\section*{Introduction}

	In the study of distributional symmetries in probability theory, the permutation group $S_n$ and the orthogonal groups $O_n$ play a central role.
	The de Finetti theorem states that a sequence of real random variables has joint distribution which is stable under each $S_n$ action if and only if it is conditionally independent and identically distributed (i.i.d.\ for short) over its tail $\sigma$-algebra.
	Similarly, the symmetry given by the orthogonal group $O_n$ induces conditionally i.i.d. centered Gaussian random variables.
	See \cite{kallenberg2006probabilistic} for details.

	In noncommutative probability theory, a probability measure space is replaced with a W$^*$-probability space $(M,\phi)$ which is a pair of a von Neumann algebra and a normal state.
	A self-adjoint operator in $M$ has a role as a random variable.
	Contrary to Kolmogorov probability theory, there are several possible notions of independence in noncommutative probability theory.
	By \cite{speicher1997universal}, there exist only three universal independences; the classical independence, the free independence and the Boolean independence.
	Free probability theory is one of the most developed noncommutative probability theory \cite{voiculescu1992free}.
	The Boolean independence appeared in \cite{von1973approach}, \cite{speicher1997boolean}.
	The Boolean one occurs only in the non-unital situations.
	Each universal independence is characterized by a family of multivariate cumulants whose index runs over  one of a category of partitions.
	Free cumulants and Boolean cumulants are determined by noncrossing partitions and interval partitions, respectively.
	By using Boolean cumulants, it can be proven that the central limit distribution of the Boolean independence is the Bernoulli distribution.

	 K{\"o}stler and Speicher have shown the free de Finetti theorem in \cite{kostler2009noncommutative}.
	The theorem states that the symmetry given by the free permutation groups $(C(S_n^+))_{n \in \N}$ induces the conditional free independence.
	The free permutation group $C(S_n^+)$ is the liberation, that is, a free analogue, of $S_n$ (See \cite{banica2012finetti} for the liberation).
	More precisely, the Hopf algebra $C(S_n^+)$ is given by eliminating the commuting relations among the generators of the Hopf algebra $C(S_n)$.
	The free permutation group is one of the free quantum group which appeared in \cite{wang1995free}, \cite{wang1998quantum}.

	An easy quantum group is one of Woronowicz's compact matrix quantum groups  which is characterized  by a tensor category of partitions in the sense of the Tannaka-Krein duality. 
	De Finetti theorems have been proven for easy quantum groups (see \cite{banica2012finetti})
	in particular easy groups $S_n$, $H_n$, $B_n$, $O_n$ and free quantum groups $C(S_n^+), C(H_n^+), C(B_n^+), C(O_n^+)$.
	 It is known that  every compact quantum group admits the unique Haar state
		 \cite{woronowicz1987compact}, and the Haar states have a main role in the 
	de Finetti theorem.

	Liu's work \cite{liu2015noncommutative} starts the research of the de Finetti theorem for the Boolean independence.
	He adds a projection  $\mathbf{P}$ to the generators of free quantum groups $C(S_n^+)$ and  defines  a quantum semigroup (in the sense of \cite{soltan2009quantum}) $\mc{B}_s(n)$ and has proven associated Boolean de Finetti theorem.
	The theorem states that the symmetry given by the family $(\mc{B}_s(n))_{n \in \N}$ characterizes the conditionally Boolean i.i.d.\ random variables.

\subsection*{Main Results}
	To develop the research of the Boolean de Finetti theorem, we are interested in finding  the Haar states on  Boolean quantum semigroups.	
	By using the Haar state, we can apply  the organized strategy for the de Finetti theorems for easy quantum groups \cite{banica2012finetti} in a similar way.		
	We define a Boolean analogue  of permutation group $S_n$ in a different form $Beq_s$.
	
	We do not prove that $Beq_s(n)$ and $\mc{B}_s(n)$ are isomorphic, but we prove that  $Beq_s(n)$ and $\mc{B}_s(n)$ admit same Haar state $h_s$.
	Moreover, we prove that the Boolean quantum semigroups $Beq_h$ on the category $I_h$ and the Boolean quantum semigroups $Beq_o$ on $I_o= I_2$ have unique Haar states $h_h$, $h_o$.  We do not prove the existence of the Haar state on Boolean quantum semigroups $Beq_b$ on  $I_b$, but we prove that of the Haar state on Boolean pr-quantum semigroups $\mc{A}_p[I_b]$.
	
	We first define the notion of categories  of interval partitions which is deeply connected with Boolean independence by Boolean cumulants.
 	By using the categories of interval partitions,
	 we induce the notion of Boolean pre-quantum semigroups $({\mathcal A}_p[D;n])_{n \in \N}$ (see Definition \ref{boo_qg}) which is a sequence of  unital $*$-algebras equipped with coproducts.
	Taking their C$^*$-completion, we define Boolean quantum semigroups $Beq_x(n)$.
	
	For a sequence of coalgebras $(\mathcal{A}(n))_{n \in \N}$, we say that $(x_j)_{j \in \N}$ is $\mathcal A$-invariant if its joint distribution is invariant under the coactions of $(\mathcal{A}(n))_{n \in \N}$.	
 	Then we show the following Boolean de Finetti theorems.
 	\begin{thm}
	Let $(M, \varphi)$ be a pair of a von Neumann algebra and a nondegenerate normal state. Assume $M$ is $\sigma$-weakly generated by self-adjoint elements $(x_j)_{j \in \N}$. 
	Let $M_{nut}$ be the non-unital tail von Neumann algebra.
	
	\begin{enumerate}

		\item[(s)] The following assertions are equivalent;
		\begin{enumerate}
			\item[(0)] The sequence $(x_j)_{j \in \N}$ is $\mc{B}_s$-invariant.
			\item[(alg)]  The sequence $(x_j)_{j \in \N}$ is $\mathcal{A}_p[I]$-invariant.
			\item[(beq)]   The sequence $(x_j)_{j \in \N}$ is $Beq_s$-invariant. 
			\item[(iid)] 
			The elements $(x_j)_{j \in \N}$ are Boolean i.i.d.\ over $M_\mathrm{nut}$.
		\end{enumerate}

		\item[(o)] The following assertions are equivalent;
		\begin{enumerate}
			\item[(alg)]  The sequence $(x_j)_{j \in \N}$ is $\mathcal{A}_p[I_2]$-invariant.
			\item[(beq)]   The sequence $(x_j)_{j \in \N}$ is $Beq_o$-invariant. 
			\item[(iid)] 
			The elements $(x_j)_{j \in \N}$ form a  $M_\mathrm{nut}$-valued Boolean centered Bernoulli family.
		\end{enumerate}
		
		\item[(h)] The following assertions are equivalent;
		\begin{enumerate}
			\item[(alg)]  The sequence $(x_j)_{j \in \N}$ is $\mathcal{A}_p[I_h]$-invariant.
			\item[(beq)]   The sequence $(x_j)_{j \in \N}$ is $Beq_h$-invariant. 
			\item[(iid)] 
			The elements $(x_j)_{j \in \N}$ are Boolean independent, and have even and identically distributions, over $M_\mathrm{nut}$.
		\end{enumerate}

	\item[(b)] The following assertions are equivalent;
	\begin{enumerate}
		\item[(alg)]  The sequence $(x_j)_{j \in \N}$ is $\mathcal{A}_p[I_b]$-invariant.
		\item[(iid)] 
		The elements $(x_j)_{j \in \N}$ form a  $M_\mathrm{nut}$-valued Boolean shifted Bernoulli family.
	\end{enumerate}
	
	\end{enumerate}

 	\end{thm}

	The common difficulty in carrying out the proof is that Boolean independence is a non-unital phenomenon.
	That is, if $M$ is a von Neumann algebra and $\phi$ is a faithful normal state on $A$, and $(M_1, M_2)$ is a pair of non-trivial von Neumann subalgebras with $1_M  \in M_1, M_2$.
	Then $(M_1, M_2)$ cannot be Boolean independent in $(M, \phi)$.
	Hence, we consider non-unital embeddings of von Neumann algebras in the arguments of Boolean independence and conditional Boolean independence.

	The main difficulty is to find the Haar states on $(Beq_x(n))_{n \in \N}$. We do that by constructing the GNS-representation of $Beq_s(n)$ on the Hilbert space $L^2(S_n)$ of $L^2$-functions on classical permutation group $S_n$.

	\subsection*{Related Works}
	In recent preprints \cite{liu2015extended} \cite{liu2015onnoncommutative}, Liu generalizes $\mc{B}_s$ in a different form from $Beq_x$ and proves generalized Boolean de Finetti theorems. His strategy does not rely on the Haar states.

	\subsection*{Organization}

	This paper consists of four sections.
	Section 1 is devoted to some preliminaries.
	In Section 2, we introduce the Boolean pre-quantum semigroups $\mathcal{A}_p[D;n]$ and the Boolean quantum semigroups $Beq_x(n)$.
	Section 3 provides a detailed exposition of the Haar functionals and the Haar states.
	In Section 4, our main results, the Boolean de Finetti type results are proved.

\section{Preliminaries}

\subsection{Partitions}
Let us review some notations related to partitions of a set.

\begin{note}\hfill
\begin{enumerate}

\item
A partition of a set $S$ is a decomposition into mutually disjoint, non-empty  subsets.  Those subsets are called blocks of the partition. We denote by $P(S)$ the set of all partitions of $S$.

\item
For a partition $\pi$ of a set $S$ and $r, s \in S$, %
  we define $r \underset{\pi}{\sim} s $ if $r$ and $s$ belong to the same block of $\pi$.
\item
Let $S, J$ be any sets and $\bfj \in$ Map $(S, J)$. We  denote by $\ker\bfj$ the partition of $S$ defined as $r \underset{\ker\bfj}{\sim} s $ if and only if $j(r)=j(s)$.
\item
  For $\pi, \sigma \in P(S)$,  we write $\pi \leq \sigma$  if each block of $\pi$ is a subset of some block of $\sigma$. The set $P(S)$ is a poset under the relation $\leq$.
 \item
  We set for $\pi, \sigma \in P(S),$
  \begin{align*}
  \delta(\pi, \sigma) := \begin{cases} 1, \text{ if } \pi = \sigma,\\
  									0,  \text{ otherwise,}
  							\end{cases}
  							\
 \zeta(\pi, \sigma) :=  \begin{cases} 1, \text{ if } \pi \leq \sigma,\\
   									0,  \text{ otherwise.}
   							\end{cases}
  \end{align*}

  \end{enumerate}
\end{note}

We introduce the M{\"o}bius function.  See \cite{nica2006lectures} for more details.

\begin{defn}[The M{\"o}bius function] \label{moeb}
Let $(P, \leq)$ be a finite poset.
The M{\"o}bius function $\mu_P \colon P^2 \ra \C$ is defined as the inverse of $\zeta$, that is, determined by  the following relations: for any $\pi, \sigma \in P$ with $\pi \not \leq \sigma$, $\mu_{P}(\pi, \sigma) = 0$, and for any $\pi, \sigma \in P$ with $\pi \leq \sigma$ ,
  \begin{align}\label{Moebious taking an interval}
  \sum_{\substack{\rho \in P \\ \pi \leq \rho \leq \sigma}} \mu_P(\pi, \rho) = \delta(\pi, \sigma), \
  \sum_{\substack{\rho \in P \\ \pi \leq \rho \leq \sigma}} \mu_P(\rho, \sigma) = \delta(\pi, \sigma),
  \end{align}

  \end{defn}

The following remark is one of the most important properties of the M{\"o}bius function to prove de Finetti theorems.

\begin{prop}
Let $Q$ be a subposet of $P$ which is closed under taking an interval, that is, if $\pi, \sigma \in Q, \rho \in P$ and $\pi \leq \rho \leq \sigma$ then $\rho \in Q$.
 Then for any $\pi, \sigma \in Q$ with $\pi \leq \sigma$, we have $
 \mu_Q(\pi, \sigma)=\mu_P(\pi, \sigma).
 $
\end{prop}
\begin{proof}
The proposition follows from the relations  \eqref{Moebious taking an interval}
\end{proof}

We define the notion of categories of interval partitions.

\begin{defn}
	A partition $\pi \in  P(k)$ is said to be \emph{an interval partition} of $[k]$ if each block contains only consecutive elements. We denote by $I(k)$ the set of all interval partitions of $[k]$.
\end{defn}

\begin{defn}
	The \emph{tensor product} $\otimes$ of partitions is defined  by horizontal concatenation.

\end{defn}

\begin{defn}\label{cat_interval}
	\emph{A category of interval partitions} is a  collection $D= (D(k))_{k \in \N}$ of subsets $D(k) \subset I(k)$, subject to the following conditions.

	\begin{enumerate}
		\item It is stable under the tensor product $\otimes$.
		\item It contains the pair partition $\sqcap$.
	\end{enumerate}
		For a  category of interval partitions $D$, let us denote
		$
			L_D :=  \{ k \in \N : {\bf 1}_k \in D(k) \},
		$
		where  ${\bf1}_k \in P(k)$ is the partition  which contains only one block $\{1, 2, \dotsc, k \}$.

\end{defn}

\begin{note}
	We denote by $I_h(k), I_b(k),$ and  $I_2(k) \subset I(k)$ the set of all interval partitions with even block size, with block size $\leq 2$, and with block size $2$ of $[k]$, respectively.
	Then each $I_x$ $(x=h,b,2)$ is a category of interval partitions. We also write $I_s = I$, $I_o = I_2$.
	Then we have $L_{I_s} = \N$, $L_{I_o} = \{2\}$, $L_{I_h} = \{2, 4, 6, \dots\}$ and $L_{I_b} = \{1,2\}$.
	\end{note}

\begin{note}
For $n \in \N$, we denote  by $l^2_n$  the standard  $n$-dimensional Hilbert space.
For $k \in \N$ and $\pi \in P(k)$, set a vector in ${l^2_n}^{\otimes k}$ by
\begin{align*}
 T^{(n)}_\pi := \sum_{\substack{\bfj \in [n]^k,\\ \pi \leq \ker\bfj}} e_\bfj,
 \end{align*}
 where $(e_i)_{i \in [n] }$ is a fixed complete orthonomal basis of $l^2_n$ and
$
  e_\bfj := e_{j_1} \otimes e_{j_2} \otimes \cdots e_{j_k}.
 $
For a category of interval partition $D$, let  $H^{D(k)}(n) \in B({l^2_n}^{\otimes k})$
be the orthogonal projection onto the subspace Span$\{T^{(n)}_\pi \mid \pi \in D(k)\}$.
We omit the index $(n)$ if there is no confusion.
We set
\begin{align*}
H^{D(k)}_{\bfi \bfj}:= \langle e_\bfi, H^{D(k)}e_\bfj \rangle.
\end{align*}

\end{note}
\begin{defn}[The Weingarten function]

   For $\pi, \sigma \in P(k)$, set the Gram matrix $G_{k,n}$ by
   $
    G_{k,n}(\pi, \sigma) := \langle T^{(n)}_\pi, T^{(n)}_\sigma \rangle $ $= n^{|\pi \vee \sigma|}$.
    Let $D$ be a category of interval partitions.
     Since the family $(T^{(n)}_\pi)_{\pi \in D(k)}$ is linearly independent for large $n$, %
  $G_{k,n}$ is invertible  for sufficiently large $n$. %
  We define the Weingarten function $W_{k,n}^D$ to be its inverse. %

\end{defn}

\begin{prop}\label{rmk_wein}
 Let $D$ be a category of interval partitions.   For any $\bfi, \bfj \in [n]^k$ and  sufficiently large $n$, we have
    \begin{align*}
    H^{D(k)}_{\bfi, \bfj} = %
    \sum_{\substack{\pi, \sigma \in D(k)\\ \pi \leq \ker \bfi \\ \sigma \leq \ker\bfj}} W_{k,n}^{D}(\pi, \sigma).
    \end{align*}

\end{prop}

\begin{proof}
This is a special case of a well-known result,  see \cite{banica2012finetti} for more details.
\end{proof}

\begin{defn}
A category $D$ of interval partitions is said to be \emph{closed under taking an interval} if
for any $k \in \N$ and $\rho, \sigma \in D(k)$, we have
\begin{align*}
 \{ \pi \in I(k) \mid \rho \leq \pi \leq \sigma \} =  \{ \pi \in D(k) \mid \rho \leq \pi \leq \sigma \}.
 \end{align*}
\end{defn}

\begin{prop}[The Weingarten estimate]\label{prop_wein_est}

Assume $D$ is closed under taking an interval.
For any $\pi, \sigma \in D(k)$,
   \begin{align*}
    n^{ | \pi | }W_{k,n}^D(\pi, \sigma) = \mu_{I(k)}(\pi, \sigma) + O(\frac{1}{n})\  (\text{as } n \ra \infty ),
   \end{align*}
\end{prop}

\begin{proof}
By \cite[Prop.3.4]{banica2012finetti}, it holds that $n^{ | \pi | }W_{k,n}^D(\pi, \sigma) = \mu_{D(k)}(\pi, \sigma) + O(1/n)$,
as $n \ra \infty$.
Since the subposet $D(k) \subset I(k)$ is closed under taking an interval, %
we have $\mu_{I(k)} = \mu_{D(k)}$,
which proves the proposition.

\end{proof}

\begin{rmk}
We call a category of interval partition $D$ is \emph{join-stable } or \emph{$\vee$-stable} if
    $\sigma \vee \rho \in D(k)$ for any $\sigma, \rho \in D(k), k \in \N$.
We see that each category of interval partitions $I_s, I_o, I_h$ is $\vee$-stable. Therefore, for $x =s,o,h$, there exists the interval partition $\max_{I_x (k)} W \in I_x(k)$  for any nonempty subset $W \subset I_x(k)$ with $W \vee W \subset W$.

However, the category $I_b$ is not $\vee$-stable.
For example,
\setlength{\unitlength}{0.5cm}
\begin{center}
\begin{picture}(18,2)
\put(-0.1,1){\uppartii{1}{1}{2}}
\put( 2, 1){\upparti{1}{1}}

\put(5, 1.3){$\vee$}

\put( 6,1){\upparti{1}{1}}
\put( 7,1){\uppartii{1}{1}{2}}

\put(11, 1.3){$=$}

\put(12, 1){\uppartiii{1}{1}{2}{3}}

\put(17, 1.3){$\not \in I_b(3).$ }

\end{picture}
\end{center}
\end{rmk}

\begin{note}
Let the index $x$ be one of $s, o, h$.
For any $k \in \N$ and $\sigma \in P(k)$,
we write
\[
\inf_{I_x} \sigma  := \max\{\pi \in I_x(k) \mid \pi \leq \sigma \}.
\]

\end{note}

\subsection{Nonunital tail von Numann algebras}
Let us define non-unital tail von Neumann algebras.
In this paper, we do not assume that an embedding of $*$-algebras, C$^*$-algebras or von Neumann algebras is unital.

\begin{defn}\label{note_intro}\hfill
\begin{enumerate}
\item For $n \in \N$, denote by $\ms P_{n}^{o}$ (resp.\,$\ms P_{\infty}^o$) the $*$-algebra of all polynomials without constant terms in noncommutative $n$-variables $X_1, \dots, X_n$ (resp.\,countably infinite many variables $(X_j)_{j \in \N}$).
\item
Let  $M$ be a von Neumann algebra. Let $(x_j)_{j \in \N}$ be a sequence of self-adjoint elements in $M$.
Denote by $\mathrm{ev}_x \colon \ms P_{\infty}^o \ra M$ the evaluation map
$\mathrm{ev}_x(X_j) = x_j.$
Let us denote by $M_\mathrm{nut}$ the non-unital tail von Neumann algebra, that is,
\begin{align*}
M_\mathrm{nut} := \bigcap_{n = 1}^\infty \overline{\mathrm{ev}_x(\ms P^o_{\geq n})}^{\sigma w},
\end{align*}
where $\ms P^o_{\geq n} := \{ f \in \ms P^o_\infty \mid f \text{ is a polynomial in variables } X_j \, (j \geq n) \}$.
\end{enumerate}
\end{defn}

We define  the notion of conditional expectations for non-unital embeddings.

\begin{defn}\label{cond_exp}
Let $\eta \colon B \hra A$ be an embedding of  $*$-algebras.
A linear map $E \colon A \ra B$ is said to be a conditional expectation with respect to $\eta$ if it satisfies the following conditions:
  \begin{enumerate}
  \item $E(x^*x) \geq 0$ for all $x \in A,$
  \item $E \circ \eta = \text{id}_B,$
  \item $E(\eta(b)x) = bE(x), E(x\eta(b))= E(x)b$ for all $b \in B, x \in A$.
  \end{enumerate}

\end{defn}

\begin{defn}
Let $A, B, \eta$ and $E$ be the same as in Definition~\ref{cond_exp}.
  Let $(a_j)_{j \in J}$ be  self-adjoint elements in $A$.
  We say $(a_j)_{j \in J}$ are identically distributed over $(E,B)$ if $E[a_i^k] = E[a_j^k]$ holds for any $i, j \in J,$ and $k \in \N$.
\end{defn}

Let us introduce the notion of conditional Boolean independence.

\begin{defn}
	Let $\eta \colon B \hra A$ be a non-unital embedding of unital $*$-algebras $A, B$ with a conditional expectation $E \colon A \ra B$. Let $1_A$ be a unit of $A$.
	Let $(x_j)_{j \in J}$ be a family of self\text{-}adjoint elements of $A$.
	Write
	\[
	B \langle x_j \rangle^o := \mathrm{Span} \bigcup_{n=1}^\infty%
	\{ b_0 x_j b_1 x_j \dots b_{n-1} x_j b_n\mid%
     b_0, \dots, b_n \in B \cup \{1_A\}	\}.
	\]
  The elements $(x_j)_{j \in J}$ are said to be \emph{Boolean independent over $(E,B)$} if
    \begin{align*}
    E[y_1 \dotsb y_k] = E[y_1] \dotsb E[y_k],
    \end{align*}
    whenever $k \in \N$, $j_1, \dotsc, j_k \in J, j_1 \neq j_2 \neq \dotsb \neq  j_k$, and $y_l \in B \langle x_{j_l} \rangle^o$, $l =1, \dotsc, k$.
\end{defn}

\begin{lem}\label{rewrite boolean indep}
 The elements $(x_j)_{j \in \N}$ are Boolean independent and identically distributed over $(E, B)$ if and only if the following holds: for any $j_1, \dots, j_k \in \N$ and $b_0, b_1, \dots b_k \in B \cup \{1_A\}$,
	\begin{align*}
		E[b_0x_{j_1}b_1x_{j_2}b_2 \dotsb x_{j_k}b_k] =  b_0 \cdot \prod_{V \in \inf_I \ker \bfj}^{\ra} E[ \prod_{l \in V}^{\ra} x_{j_l}b_l].
	\end{align*}

\end{lem}
\begin{proof}
	For  $r,s \in [n]$, $r \sim^{\inf_I \ker \bfj} s$ if and only if  $r$ and $s$ are consecutive elements and $j_r=j_s$. By the linearlity of $E$, we have the claim.
\end{proof}

\subsection{Boolean cumulants}

	In operator-valued free probability,
operator-valued cumulants characterize the conditional free independence (see \cite{nica2006lectures} \cite{speicher1998combinatorial}).
	We introduce some properties of the operator-valued Boolean cumulants. They combinatorially characterize conditional Boolean independence.
	Single variate Boolean cumulants are defined in \cite{speicher1997boolean}.
	As far as the author knows, multivariate Boolean cumulants first appeared in \cite{lehner2004cumulants}.

	Throughout this section, we suppose $B \subset A$ is an embedding of $*$-algebras (not necessarily unital) with a normal conditional expectation $E$.

\begin{note}\hfill
	\begin{enumerate}
	\item
	Let $(S, \leq)$ be a finite totally ordered set and we write $S = \{s_1 < s_2 < \dots < s_n\}.$
	For a family $(a_s)_{s\in S}$ of elements in $M$, we denote by $\prod_{s \in S}^{\ra}a_s$ the ordered product  $\prod_{s \in S}^{\ra} a_s = a_{s_1} \dotsb a_{s_n}$.

	\item
		For an interval partition $\pi$ and blocks $V, W \in \pi$, we write $V \leq W$ if $k \leq l$ for any $k \in V$ and $l \in W$.  The set $\pi$ is a totally ordered set under the relation $\leq$.

	\end{enumerate}

\end{note}

\begin{defn}\label{def_cumulant}
	Let us define $B$-valued multilinear functions $K^E_\pi : A^n \ra B $ $(\pi \in I(k), k \in \N)$ inductively by the following three relations:
	\begin{enumerate}
		\item For $k \in \N$ and $y_1, \dotsc, y_k \in M$,
		$
		E[y_1 \dotsb y_k]  = \sum_{\pi \in I(k)} K_\pi^E [y_1, \dotsc, y_k].
		$
		\item For $k \in \N$ and $\pi \in I(k)$,
		$
		K_{\pi}^E [y_1, \dotsc, y_k] = %
		\prod_{V \in \pi}^{\ra}K_{(V)}^E [y_1, \dotsc, y_k].
		$
		\item  For $\pi \in I(k)$ and $V \in \pi$, $K_{(V)}^E [y_1, \dotsc, y_k] := K_{{\bf1}_m}^E [y_{j_1}, \dotsc, y_{j_m}]$ where $V= \{j_1 < j_2 < \cdots < j_m\}$.
	\end{enumerate}
We call them \emph{Boolean cumulants} with respect to $E$. We write $K_n^E = K^E_{{\bf1}_n}$ for $n \in \N$.

\end{defn}
\begin{prop}\label{moments-cumulants formula}

For $\pi \in I(k),$ $y_1, \dotsc, y_k$ and $k \in \N$,  set
$
E^{\pi}[y_1, \dotsc, y_k] := \prod_{V \in \pi}^{\ra} E[ \prod_{j \in V}^{\ra}y_j].
$
Then for $\pi \in P(k)$, $y_1, \dotsc, y_k \in M$ and $k \in \N$,
\begin{align*}
	E^{\pi}[y_1, \dotsc, y_k] =%
	 \sum_{\substack{\sigma \in I(k) \\ \sigma \leq \pi}}K_{\sigma}^E [y_1, \dotsc, y_k].
\end{align*}
Hence we have
$K_{\pi}^{E} [y_1, \dotsc, y_k]  %
= \sum_{\substack{\sigma \in I(k) \\ \sigma \leq \pi}}E^{\sigma}[y_1, \dotsc, y_k] \mu_{I(k)}(\sigma, \pi)$.
\end{prop}

\begin{proof}
The proof is a straight induction on $|\pi|$.
\end{proof}

The conditional Boolean independence  can be characterized by vanishing of mixed cumulants.

\begin{thm}\label{thm_indep} Let $(x_j)_{j \in J}$ be a family of self-adjoint elements in $A$.
	Then $(x_j)_{j \in J}$ are Boolean independent identically distributed with respect to $E$ if and only if
	\begin{align*}
	E[b_0x_{j_1}b_1x_{j_2}b_2 \dotsb x_{j_k}b_k] = %
	\sum_{\substack{\pi \in I(k) \\ \pi \leq \ker\bfj}}K_\pi^E [b_0 x_1 b_1, x_1b_2, \dotsc, x_1b_k]
	\end{align*}
	for any $b_1, \dotsb, b_k \in B \cup \{ 1_A \}$,  $\bfj \in J^k, k \in \N$.
\end{thm}

\begin{proof} We have
   	\[
	b_0 \cdot \prod_{V \in \inf_I \ker \bfj}^{\ra} E[ \prod_{l \in V}^{\ra} x_1b_l]%
	= b_0 E^{\inf_I \ker \bfj}[x_1b_1, x_2b_2, \dots, x_1b_k]%
	=\sum_{\substack{\pi \in I(k) \\ \pi \leq \inf_I \ker\bfj}}K_\pi^E [b_0 x_1 b_1, x_1b_2, \dotsc, x_1b_k].
	\]
    We see that $\{ \pi \in I(k) \mid \pi \leq \inf_I \ker\bfj \} = \{ \pi \in I(k) \mid \pi \leq  \ker\bfj \}$.
    Lemma~\ref{rewrite boolean indep} completes the proof.
\end{proof}

\begin{defn}\label{ber}
Let $x$ be a self-adjoint element in $(M, E)$.
   \begin{enumerate}
   \item The element $x$ is said to  \emph{have centered Bernoulli distribution} if
   for any $b_1, \dotsc, b_{k-1} \in N \cup \{1_M\}$ and $k \in \N$,
   \begin{align*}
      E[xb_1xb_2 \dotsb b_{k-1}x] = \sum_{\pi \in I_2(k)}K_\pi^E [xb_1, xb_2, \dotsc, x].
   \end{align*}
  We see immediately that  if $N=\C1_M$, $x$ has centered Bernoulli distribution if and only if that is $(\delta_{\sigma} + \delta_{-\sigma})/2 $ where $\sigma:=\sqrt{K^E_2[x,x]}$.

   \item The element $x$ is said to  \emph{have shifted Bernoulli distribution} if
   for any $b_1, \dotsc, b_{k-1} \in N \cup \{1_M\}$ and $k \in \N$,
   \begin{align*}
   E[xb_1xb_2 \dotsb b_{k-1}x] = \sum_{\pi \in I_b(k)}K_\pi^E [xb_1, xb_2, \dotsc, x].
   \end{align*}
	We check easily that   if $N=\C1_N$, $x$ has shifted Bernoulli distribution with $K^E_1[x]= \mu$ and $K^E_2[x, x]=\sigma^2$ if and only if its distribution is
   \begin{align*}
    \mathrm{Ber}(\mu,\sigma^2):=\frac{\alpha \delta_{\alpha} + \beta \delta_{-\beta}}{\alpha + \beta},
   \end{align*}
   where $\alpha, \beta >0$, and $(\alpha,-\beta)$ is the pair of distinct  solutions of the  quadratic equation $Z^2-\mu Z - \sigma^2 = 0$ in the variable $Z$. Its $n$-th moment is given by
   $
   E[x^n]= \left( \alpha^{n+1}- (-\beta)^{n+1} \right) / (\alpha + \beta).
   $
  \end{enumerate}
\end{defn}

 A Bernoulli distribution is the central limit distribution of Boolean i.i.d.\,self-adjoint elements (see \cite{speicher1997boolean}).
 Hence the Bernoulli distribution is the Boolean analogue of  Gaussian distribution.

\section{Boolean analogues of easy quantum groups}

\subsection{Boolean quantum semigroups}\hfill

In this section we introduce the notions of Boolean quantum semigroups on categories of interval partitions.

\begin{defn}\label{blockwise}
	For   a category $D$ of interval partitions, consider the following three conditions.
	\begin{enumerate} \renewcommand{\labelenumi}{(D\arabic{enumi})}
		\item \label{D1} 		It is  \emph{block-stable}, which  means that for any $k \in \N$,
           \begin{align*}
			D(k) = \{ \pi \in D(k) \mid \{V\} \in D(|V|),\ V \in \pi \}.
			   \end{align*}
		\item \label{D2} It is closed under taking an interval.
		\item \label{D3} It \emph{has enough patitions}, which means that for $l \in \N $, it holds that $ D(l) \neq \emptyset $ if there is  $k \in L_D$ with
		$D(k+l) \neq \emptyset $.

	\end{enumerate}
	We say that $D$ is \emph{blockwise} if it satisfies  (D1)--(D3).
\end{defn}

\begin{example}
	Categories $I_s, I_o,I_h, I_b$ of interval paritions  are blockwise.

\end{example}

\begin{defn}\label{boo_qg}
	Let $D$ be a blockwise category of partitions.
	Denote by $\mathcal{A}[D;n]$ the non-unital $\ast$-algebra generated by self-adjoint elements $u_{ij}^{(n)} (1 \leq i,j \leq n)$ and an orthogonal projection $p^{(n)}$ with the following relations: for any $k \in L_D$ and $\bfi, \bfj \in [n]^k$,
	\begin{align*}
	\sum_{i =1}^n u_{i j_1}^{(n)} \dotsb u_{i j_k}^{(n)} p^{(n)} %
	&= \begin{cases} p^{(n)}, & j_1=\dotsb=j_k, \\ 0, & \text{otherwise,}\end{cases} \\
	\sum_{j =1}^n u_{i_1 j}^{(n)} \dotsb u_{i_k j}^{(n)} p^{(n)} %
	&= \begin{cases} p^{(n)}, & i_1=\dotsb=i_k, \\ 0, & \text{otherwise.}\end{cases} \\
	\end{align*}
	If there is no confusion, we omit the index $(n)$ and simply write $u_{i,j}$ and $p$.
	There is a linear map $\Delta \colon {\mathrm A}[D;n] \ra {\mathrm A}[D;n] \otimes {\mathrm A}[D;n]$ with
	\begin{align*}
	\Delta(p) :=p \otimes p, \ \
		\Delta(u_{ij}) := \sum_{k = 1}^n u_{ik} \otimes u_{kj} \  ( i,  j = 1, 2, \dots, n).
	\end{align*}
	It is easy to check that $\Delta$ is a coproduct, that is, the following  holds:
	\begin{align*}
	(\mathrm{id} \otimes \Delta) \Delta = (\Delta \otimes \mathrm{id}) \Delta.
	\end{align*}
	Set a linear linear map $\epsilon \colon {\mathrm A}[D;n] \ra \C$ by
$	\epsilon(u_{ij}) = \delta_{ij}, \ \epsilon(p)= 1.
$
	We have $(id \otimes \epsilon) \Delta = id = (\epsilon \otimes id) \Delta.$
	Hence  ${\mathrm A}[D;n]$ is a coalgebra with the coproduct $\Delta$ and the counit $\epsilon$.
	We define a sequence of unital $*$-algebra equipped with coproduct by
	\begin{align*}
		\mathcal{A}_p[D;n] := p{A}[D;n]p
	\end{align*}
	We call $(\mathcal{A}_p[D;n])_{n \in \N}$ \emph{the Boolean pre-quantum semigroups on $D$}.
\end{defn}

\begin{defn}
We call the sequences of pairs $((Beq_x(n), \Delta))_{n \in \N}$ defined by the following   \emph{the Boolean quantum semigroups on $D$} for %
 $I_x$ $(x = s, o, h, b)$.

  \begin{enumerate}
  \item   For $a \in \mathcal{A}[I_x;n]$, we set
     \begin{align*}
       || a || :=   \sup \{  ||\pi(a)||
       \mid   \pi \text{ is a}*\text{-representation of } \mathcal{A}[I_x; n],\ \pi(p)=1, \ ||(\pi(u_{ij}))_{i,j}||_n \leq 1 \},
     \end{align*}
     where $|| \cdot ||_n $ is the operator norm on $B(H) \otimes M_n(\C)$ for each $*$-representation $(\pi, H)$.
     Since there is the  $*$-representation $\pi \colon A[I;n] \ra C(S_n) \subset B(H)$ defined by $\pi(u_{ij})(\sigma) = \delta(\sigma(i), j)$ $( \sigma \in S_n )$,
     we obtain $ 0 \leq || a ||$ $< \infty $.
     Hence $|| \cdot ||$ is a C$^*$-seminorm on $\mathcal{A}[I_x; n] $.
  \item    Let $B$ be the C$^*$-completion of
      $\mathcal{A}[I_x; n] / \langle  || \cdot || = 0 \rangle$.
     We define \emph{the Boolean quantum semigroup on $I_x$ of $n$} by
     \begin{align*}
       Beq_x(n) := p B p.
     \end{align*}
\item    We denote by $\iota_n$ the unital $*$-hom  $\mathcal{A}_p[I_x;n] \ra Beq_x(n)$ which is the restriction of the $*$-hom $\mathcal{A}[I_x;n] \ra Beq_x(n)$ determined by $\iota_n(u_{ij}) = [u_{ij}] \ (i, j \in [n])$, $\iota_n(p) = [p]$.
  By abuse of notation, we use same symbols $u_{ij}$, $p$ for the generators $[u_{ij}]$, $[p]$ of $Beq_x(n)$.

\item For any $*$-representation $\pi$ of $\mathcal{A}[I_x;n]$ with $|| (\pi (u_{ij}) )_{ij} ||_n \leq 1$,
  we obtain $|| (\pi(\Delta u_{ij})_{ij} ||_n \leq 1$.
  Hence we can extend the domain of $\Delta$, %
  that is, there is a unique bounded $*$-hom %
  $\bar{\Delta} \colon Beq_x(n) \ra Beq_x(n) \otimes_\mathrm{min} Beq_s(n)$
  with $\bar{\Delta}(u_{ij}) = \sum_{s \in [n]} u_{is} \otimes u_{sj}$ and  %
  $\bar{\Delta}(p)=p \otimes p$.
  We simply denote by $\Delta$ the bounded $*$-hom $\bar{\Delta}$  if there is no confusion.
  It is easy to check that $\Delta$ is a coproduct of $Beq_x(n)$.

  \end{enumerate}

\end{defn}

\begin{lem} \label{redefine bool_eqg}
Let  the index $x$ be one of $s,o,h,b$. Then for any $k, n \in \N$ and $\pi \in I_x(k)$,  we have
  \begin{align*}
  \sum_{\substack{\bfi \in [n]^k \\ \pi \leq \ker \bfi}} u_{i_1j_1}^{(n)} \dotsb u_{i_kj_k}^{(n)} p^{(n)} %
  = \begin{cases} p^{(n)}, & \pi \leq \ker \bfj, \\ 0, & \text{otherwise},\end{cases} \
  \sum_{\substack{\bfj \in [n]^k \\ \pi \leq \ker \bfj}} u_{i_1j_1}^{(n)} \dotsb u_{i_kj_k}^{(n)} p^{(n)} %
  = \begin{cases} p^{(n)}, & \pi \leq \ker \bfi,\\ 0, & \text{otherwise}.\end{cases}
  \end{align*}
\end{lem}

\begin{proof}
The proof is induction on $|\pi|$.
\end{proof}

\begin{rmk}
We denote by $P_{ij} \in B(L^2(S_n) \ (i, j \leq n)$ the generators of $C(S_n)$, where $P_{ij}(\sigma) = \delta_{i, \sigma(j)}$ $(\sigma \in S_n)$.
We see at once that there is a $*$-representation $Beq_s \ra B(L^2(S_n))$ which maps
$u_{ij}$ to $P_{ij}$ and $P$ to $1$.

In  Section~\ref{Haar states on Beq_x}, we prove that there is the other $*$-representation on $L^2(S_n)$ (see Notation~\ref{note pi s}, Propositon~\ref{repsentation pi_s relation}).
In the construction, we use $P_{ij}$ in the different way.
Let $\hat{P}_{ij}$ (resp.\ $\hat{1}$) be the image of $P_{ij}$ (resp.\ $1$) with respect to the standard inclusion $C(S_n) \hra L^2(S_n)$. The $*$-representation maps $u_{ij}$ (resp.\ $p$) to the one dimensional projection onto the closed subspace $\C \hat{P}_{ij}$ (resp.\ $\C \hat{1}$) $\subset L^2(S_n)$.
Furthermore, we show that $Beq_s$ admits the unique Haar state  and that this $*$-representation is the GNS-representation of the Haar state (see Theorem~\ref{Haar state h_s existence}).

\end{rmk}

Next we consider coactions on the $*$-algebra of noncommutative polynomials without constant terms.
\begin{defn}
Let $\mathcal{A}$ be a unital $*$-algebra equipped with a coproduct $\Delta$.
For any $*$-algebra $\mathcal{P}$, a $*$-preserving linear map $T \colon \mathcal{P} \ra \mathcal{P} \otimes \mathcal{A}$ is said to be \emph{a linear coaction on $\mathcal{P}$} if we have
\begin{align*}
		(T \otimes \mathrm{id}) \circ T = (\mathrm{id} \otimes \Delta) \circ T.
		\end{align*}

\end{defn}

\begin{note} Let $D$ be a blockwise category of interval partitions.
	\begin{enumerate}
		\item For $m,n \in \N$ with $m \geq n$, we define a $*$-hom $r_{nm} \colon \mathcal{A}[D;m]\ra \mathcal{A}[D;n]$ by
		\begin{align*}
		r_{nm}(u_{ij}^{(m)}) := \begin{cases} u_{ij}^{(n)}, & i, j \leq n, \\ \delta_{ij}1_{\mathcal{A}[D;n]}, & \text{otherwise,}\end{cases}
		\hspace{5mm}  r_{nm}(p^{(m)}) := p^{(n)}.
		\end{align*}
		\item
		Define a linear map $\Lambda_n \colon \ms P_{n}^{o} \ra \ms P_{n}^{o} \otimes \mathcal{A}_p[D;n]$ by
		\begin{align*}
		\Lambda_n(X_{j_1} \dotsb X_{j_k}) := \sum_{\bfi \in [n]^k} X_{i_1} \dotsb X_{i_k} \otimes pu_{i_1j_1} \dotsb u_{i_kj_k}p.		\end{align*}
		We define a linear map $\Psi_n \colon \ms P_{\infty}^o \ra \ms P_{\infty}^o \otimes \mathcal{A}_p[D;n]$ by
		\begin{align*}
		\Psi_n (f) := (\mathrm{id} \otimes  r_{nm} )\circ \Lambda_m (f),
		\end{align*}
		for $f \in \ms P_{m}^{o}\subset \ms P_{\infty}^o$.
		Then  by a direct calculation,  each $\Psi_n$ is a linear coaction of $\mathcal{A}_p[D;n]$ on $\ms P_{\infty}^o$.

		\item
		We define a coaction $\Phi_n$ of $Beq_x(n)$ on $\ms P_{\infty}^o$ by
		\begin{align*}
			\Phi_n := (\mathrm{id} \otimes \iota_n ) \circ \Psi_n.
		\end{align*}
	\end{enumerate}

\end{note}

\begin{defn}

	Let $(M, \varphi)$ be  a pair of a von Neumann algebra and a state.
For any sequence $(x_j)_{j \in \N}$ of self-adjoint elements in $M$, 	we say that its joint distribution is
 \emph{$\mathcal{A}_p[D]$-invariant} if it is invariant under the coactions of $(\mathcal{A}[D;n])_{n \in \N}$, that is,
for any $n \in \N$,
\begin{align*}
(\varphi \circ \mathrm{ev}_x \otimes \mathrm{id} )\circ \Psi_n = \varphi \circ \mathrm{ev}_x \otimes p.
\end{align*}
We also say that it is $Beq_x$-invariant if  for any $n \in \N$,
\begin{align*}
(\varphi \circ \mathrm{ev}_x \otimes \mathrm{id} )\circ \Phi_n = \varphi \circ \mathrm{ev}_x \otimes p.
\end{align*}

\end{defn}

It is clear that $\mathcal{A}[I_x]$-invariance implies $Beq_x$-invariance.

\subsection{Relations with Liu's Boolean quantum semigroups}\hfill

		We introduce Liu's boolean permutation quantum semigroup defined in \cite{liu2015noncommutative}.
		 Let $B_s(n)$ be  the universal unital C$^*$-algebra generated by projections $\mathbf{P}, U_{i,j} (i, j = 1, \dots, n)$ and  relations such that
		\begin{align*}
		\sum_{i = 1}^n U_{ij} \mathbf{P} &= \mathbf{P},\ j = 1, \dots, n,\\
		U_{i_1 j}U_{i_2 j} &= 0,
		\text{ if } i_1 \neq  i_2, \text{ for any } j = 1, \dots, n,\\
		U_{i j_1}U_{i j_2} &= 0,
		\text{ if } j_1 \neq j_2, \text{ for any } i = 1, \dots, n.
		\end{align*}
	By \cite[Lemma~3.3]{liu2015noncommutative}, we have
	\begin{align*}
		\sum_{j = 1}^n U_{ij} \mathbf{P} = \mathbf{P},\ i = 1, \dots, n.
	\end{align*}
	We see that $B_s(n)$ admits a coproduct $\Delta$ determined by $	\Delta(\mathbf{P}) :=\mathbf{P} \otimes \mathbf{P},
			\Delta(U_{ij}) := \sum_{k = 1}^n U_{ik} \otimes U_{kj} \  ( i,  j = 1, 2, \dots, n).
	$
	Then let us introduce Liu's boolean permutation quantum semigroup.
	\begin{defn}
	We set $\mc{B}_s(n) = \mathbf{P}B_s(n)\mathbf{P}$, and we call $(\mc{B}_s(n), \Delta)$ \emph{the boolean permutation quantum semigroup} of $n$.
	\end{defn}
	We can check that each $\mc{B}_s(n)$ is a quantum semigroup in the sense of  Soltan \cite{soltan2009quantum}.
	\begin{lem}\label{hom to B_s}
	There is a $*$-hom $\alpha \colon Beq_s(n) \ra \mc{B}_s(n)$ with $\alpha(u_{ij}) = U_{ij} $ $(i, j \leq n)$ and $\alpha(p) = \mathbf{P}$.
	\end{lem}
		\begin{proof}
		We see that  for any $k \in \N$ and $\bfi, \bfj \in [n]^k$,
			\begin{align*}
			\sum_{i =1}^n U_{i j_1} \dotsb U_{i j_k} \mathbf{P} %
			= \begin{cases} \mathbf{P}, & j_1=\dotsb=j_k, \\ 0, & \text{otherwise,}\end{cases} \hspace{5mm}
			\sum_{j =1}^n U_{i_1 j} \dotsb u_{i_k j} \mathbf{P} %
			= \begin{cases} \mathbf{P}, & i_1=\dotsb=i_k, \\ 0, & \text{otherwise.}\end{cases} \\
			\end{align*}
		This completes the proof.
		\end{proof}

	\begin{note}
	We set a linear map $L_n \colon \ms P_{n}^{o} \ra \ms P_{n}^{o} \otimes \mc{B}_s(n)$ by
	$L_n(X_{j_1} \dotsb X_{j_k}) := \sum_{\bfi \in [n]^k} X_{i_1} \dotsb X_{i_k}$ $\otimes \mathbf{P}U_{i_1j_1} \dotsb U_{i_kj_k} \mathbf{P}$.
			We set a linear map $\mathbf{L}_n \colon \ms P_{\infty}^o \ra \ms P_{\infty}^o \otimes \mc{B}_s(n)$ by
			$ 
			\Psi_n (f) := (\mathrm{id} \otimes  r_{nm} )\circ \Lambda_m (f),
			$ 
			for $f \in \ms P_{m}^{o}\subset \ms P_{\infty}^o$.
			Then  by a direct calculation,  each $\mathbf{L}_n$ is a linear coaction of $\mc{B}_s(n)$ on $\ms P_{\infty}^o$.

	\end{note}

Let $(M, \varphi)$ be a von Neumann algebra and a nondegenerate normal state and $(x_j)_{j \in \N} $ be a sequence of self-adjoint elements in $M$.
We may assume $M \subset B(H)$,   and $\varphi$ is implemented by $\Omega \in H$, which is a cyclic vector for $M$. %
We suppose  that $\mathrm{ev}_x(\ms P_{\infty}^o)$ is $\sigma$-weakly dense in $M$, where  $\mathrm{ev}_x$ is the evaluation map.
	\begin{note}
		We say that $(x_j)_{j \in \N}$  is $\mc{B}_s$-invariant if  for any $n \in \N$,
			$(\varphi \circ \mathrm{ev}_x \otimes \mathrm{id} )\circ \mathbf{L}_n = \varphi \circ \mathrm{ev}_x \otimes \mathbf{P}.
			$
		\end{note}

	\begin{lem}\label{implies B_s-invariant}
	Assume $(x_j)_{j \in \N}$ is $\mathcal{A}_p[I_x]$-invariant or $Beq_x$-invariant for one of
	$x=s,o,h,b$. Then it is $\mc{B}_s$-invariant.
	\end{lem}
	\begin{proof}
	This follows immediately from Lemma~\ref{hom to B_s}.
	\end{proof}
	We review that $\mc{B}_s$-invariance implies the existence of the normal conditional expectation onto the non-unital tail von Neumann algebra.
	Assume that  $(x_j)_{j \in \N}$ is $B_s$-invariant.
	Then by \cite[Lemma~6.4]{liu2015noncommutative} ,  for $a \in \mathrm{ev}_x(\ms P_\infty^o)$,
	$
	E[a]:= \sigma \text{w-}\lim_{n \ra \infty} \text{sh}^n (a)
	$
	is  well-defined, $E[a] \in M_\mathrm{nut}$ and $E$ is state-preserving.
	By \cite[Lemma~6.7]{liu2015noncommutative} , we have	for any $a,b,c \in \mathrm{ev}_x(\ms P_\infty^o)$,
	$
		\langle E[a]b\Omega, c\Omega \rangle = $
		$			\langle a E_\mathrm{nut}[b]\Omega, E[c]\Omega \rangle.
	$
	By \cite[Lemma~6.8]{liu2015noncommutative}, we can define $E_\mathrm{nut} : M \ra M_\mathrm{nut}$
	by
	\begin{align}\label{defn of E_nut}
 E_\text{nut}[y] := \sigma w\text{-} \lim_{n \ra \infty}E[y_n],
	\end{align}
 where $(y_n)$ is a bounded sequence in $\mathrm{ev}_x(\ms P_\infty^o)$ with $\sigma w \text{-}\lim_{n \ra \infty} y_n = y$.
	By \cite[Lemma~6.9]{liu2015noncommutative}, $E_\mathrm{nut}$ is normal.
	By \cite[Lemma~6.10]{liu2015noncommutative}, $E[b]=b$ for any $b \in M_\mathrm{nut}$.
	By \cite[Lemma~6.11]{liu2015noncommutative} and since $E$ is normal, it holds that for any $y, z_1, z_2 \in M$,
		\begin{align}\label{E=EE}
		\langle E_\tl[y]z_1\Omega, z_2 \Omega \rangle = %
		\langle y E_\mathrm{nut}[z_1]\Omega, E_\tl [z_2]\Omega \rangle.
		\end{align}
	In particular, $\varphi \circ E_\tl= \varphi$.
	By \cite[Lemma~6.12]{liu2015noncommutative}, $E[by]=bE[y]$, $E[yb]=E[y]b$  for any $b \in  M_\mathrm{nut}$ and $y \in M$. Hence $E_\mathrm{nut}$ is a normal conditional expectation onto $M_\mathrm{nut}$, which is state-preserving.

	\begin{prop}\label{bs-inv}

		Assume that   $(x_j)_{j \in \N}$ is $\mc{B}_s$-invariant.
		Let $E_\mathrm{nut}: M \ra M_\mathrm{nut}$ be the conditional expectation defined by \eqref{defn of E_nut}.
		Set $e_\mathrm{nut} \in B(H)$ be the orthogonal projection onto %
	the closed subspace $\overline{M_\mathrm{nut}\Omega}$.
		Then it holds that
		\[
		E_\mathrm{nut}[y] = e_\mathrm{nut} \, y \, e_\mathrm{nut}  \ (y \in M).
		\]
		In particular,
		$
		M_\mathrm{nut} = e_\mathrm{nut} M e_\mathrm{nut}.
		$
	\end{prop}

	\begin{proof}
		Let $b \in M_\tl$, $y \in M$. As  $E_\tl [b^*y]=b^*E_\tl [y]$,
		$
		\langle b \Omega, (y - E_\tl [y]) \Omega \rangle %
		= \langle  \Omega, (b^*y - E_\tl [b^*y]) \Omega \rangle%
		=0.
		$
		Hence $(b^*y - E_\tl [b^*y]) \Omega \in \overline{M_\tl \Omega}$,
		 $ e_\tl y \Omega = E_\tl [y] \Omega$.
		By \eqref{E=EE}, for any $y \in M$, $a, b \in \mathrm{ev}_x(\ms P_\infty^o)$,
		 \[
		\langle E_\mathrm{nut}[y]a\Omega, b \Omega \rangle %
		=\langle y E_\mathrm{nut}[a]\Omega, E_\mathrm{nut}[b]\Omega \rangle%
		= \langle y e_\tl a \Omega, e_\mathrm{nut} b \Omega \rangle%
		= \langle e_\tl y e_\tl a \Omega,  b \Omega \rangle.
		\]
		Since  the subspace $\mathrm{ev}_x(\ms P_\infty^o)\Omega$ is dense in $H$, it holds that $E_\tl [y] = e_\tl y e_\tl$. As $E_\tl [M]=M_\tl$, it holds that $M_\tl = e_\tl M e_\tl$.
	\end{proof}

	\begin{cor}
	Assume $(x_j)_{j \in \N}$ is $\mathcal{A}_p[I_x]$-invariant or $Beq_x$-invariant for one of
		$x=s,o,h,b$.%
	Then
	  $
	  E_\mathrm{nut}[y]:= e_\mathrm{nut} \, y \, e_\mathrm{nut}
	  $
	  $ \ (y \in M)
	  $
	is a nondegenerate normal conditional expectation onto $M_\mathrm{nut}$ with respect to the embedding $M_\mathrm{nut} \subset M$.
	\end{cor}

	\begin{proof}
 	This follows from Lemma~\ref{implies B_s-invariant} and Proposition \ref{bs-inv}.
	\end{proof}

\section{Haar functionals and Haar states}

\subsection{Haar functionals on ${\mathcal A}[D;n]$}\hfill

 At first,  we construct a linear functional with an invariance property on ${\mathcal A}[D;n]$  instead of a Haar state.

\begin{defn}\label{definition Haar}
	Let $\mathcal{A}$ be a unital $*$-algebra. Assume $\mathcal{A}$ is equipped with a coproduct $\Delta$.
	A linear functional $h$ (resp.\ a state) on $\mathcal{A}$ is called \emph{a Haar functional} (resp.\ \emph{a Haar state}) if it satisfies the following Haar invariance property:
	\begin{align}\label{inv Haar}
	  (\mathrm{id} \otimes h)  \Delta  = h(\cdot)1_\mathcal{A} = (h \otimes \mathrm{id}) \Delta.
	\end{align}

\end{defn}

\begin{prop}\label{uniqueness Haar}
	Under the assumption of $\mathcal{A}$ in Definition \ref{definition Haar},
	the unital Haar functional on $\mathcal{A}$ is unique if it exists.

\end{prop}
\begin{proof}
	Assume that $g, h$  are  unital Haar linear functionals on $\mathcal{A}$.
	Combining invariant properties,  for any $a \in \mathcal{A}$ we obtain 
	$
	(h \otimes g) \Delta (a)  %
	=(h \otimes \mathrm{id})(\mathrm{id} \otimes g) \Delta (a)%
	= (h \otimes \mathrm{id})( 1_\mathcal{A} \otimes g(a))%
	= g(a).
	$
	Similarly,
	$
	(h \otimes g) \Delta (a)  %
	=(\mathrm{id} \otimes g)(h \otimes \mathrm{id}) \Delta (a)%
	= (\mathrm{id} \otimes g)( h(a) \otimes 1_\mathcal{A})%
	= h(a).
	$
	This completes the proof.
\end{proof}

\begin{note}
Let $D$ be a category of interval partitions.

\begin{enumerate}
\item Set
 \begin{align*}
 V_n^D := \mathrm{Span}(\{p\}\cup\{pu_{i_1j_1} \dotsb u_{i_kj_k}p \ | \ \bfi,\bfj\in[n]^k, k\in\N\}) \subset \mathcal{A}_p[D;n].
 \end{align*}
We  see at once that $\Delta(V_n^D) \subset V_n^D \otimes V_n^D$.

	\item We write $u_{\bfi \bfj} = u_{i_1 j_1} \cdots u_{i_k j_k}$ for $\bfi, \bfj \in [n]^k, k \in \N$.
	Fix a complete orthonormal basis $ \{e_i \}_{i \in [n] }$ of the standard $n$ dimensional Hilbert space $l_{2}^n$. Set $e_\bfi := e_{i_1} \otimes \dots \otimes e_{i_k}$ for $\bfi \in [n]^k$.

\item
We denote by $\Lambda_{n}^{k}$  the linear map ${l^2_n}^{\otimes k} \ra {l^2_n}^{\otimes k} \otimes V_n^D$ defined by
 \begin{align*}
 \Lambda_n^k(e_\bfj) := \sum_{\bfi \in [n]^k} e_\bfi \otimes pu_{\bfi \bfj}p.
 \end{align*}
By  a direct calculation, $\Lambda_n^k$ is a linear coaction of  $V_n^D$, that is,
 \begin{align*}
 (\mathrm{id} \otimes \Delta)  \Lambda_n^k = (\Lambda_n^k \otimes \mathrm{id})  \Lambda_n^k.
 \end{align*}
\item
Let  Fix($\Lambda_n^k$) denote the invariant subspace of the coaction $\Lambda_n^k$, that is,
\begin{align*}
\mathrm{Fix} (\Lambda_n^k) := \{ \xi \in {l^2_n}^{\otimes k} \mid \Lambda_n^k (\xi) = \xi \otimes p \}.
\end{align*}

\end{enumerate}

\end{note}

\begin{lem}\label{Haar inv under prod}
	Let $g$ be a  functional on $ {\mathcal A}_p[I;n]$.
	Assume $g |_{V_n^D} $ satisfies the Haar invariance property and $g(apb) = g(a)g(b)$ for any $a, b \in V_n^D$.
	Then $g$ is a Haar functional.
\end{lem}

\begin{proof}
	For any $k,l \in [n]$ and for any multi-indices $\bfi^{(1)}$, $\bfi^{(2)}, \dots, \bfi^{(l)}$, $\bfj^{(1)}$, $\bfj^{(2)}, \dots, \bfj^{(l)}$ $\in [n]^k$,
	\begin{align*}
	&(\mathrm{id} \otimes h) \Delta( p u_{\bfi^{(1)} \bfj^{(1)}} p u_{\bfi^{(2)} \bfj^{(2)}} p \cdots p u_{\bfi^{(l)} \bfj^{(l)}} p )\\
	&= \sum_{\bfs^{(1)}, \cdots, \bfs^{(l)} \in [n]^k}%
	p u_{\bfi^{(1)} \bfs^{(1)}} p u_{\bfi^{(2)} \bfs^{(2)}} p \cdots p u_{\bfi^{(l)} \bfs^{(l)}} p%
	\cdot h( p u_{\bfs^{(1)} \bfj^{(1)}} p)h( p u_{\bfs^{(2)} \bfj^{(2)}} p) \cdots h( p u_{\bfs^{(l)} \bfj^{(l)}} p)\\
	&= (\mathrm{id} \otimes h)  \Delta (p u_{\bfs^{(1)} \bfj^{(1)}} p) %
	\cdot (\mathrm{id} \otimes h)  \Delta (p u_{\bfs^{(2)} \bfj^{(2)}} p)%
	\cdots (\mathrm{id} \otimes h) \Delta( p u_{\bfs^{(l)} \bfj^{(l)}} p).
	\end{align*}
	This finishes the proof by using
 the Haar invariance on $V_n^D$.
\end{proof}

\begin{lem}

	For any $k, n \in \N$, $\pi \in D(k)$, and $\bfi \in [n]^l$,
	\begin{align*}
		\Lambda_n^{k+l}(T_\pi \otimes e_\bfi) = T_\pi \otimes \Lambda_n^l(e_\bfi), \
		\Lambda_n^{l+k}( e_\bfi \otimes T_\pi ) =  \Lambda_n^l (e_\bfi) \otimes T_\pi.
	\end{align*}
\end{lem}
\begin{proof}

\begin{align*}
	\Lambda_n^{k+l}(T_\pi \otimes e_\bfi) &= %
	\sum_{\substack{\bfj \in [n]^k\\ \pi \leq \ker\bfj}}%
	\Lambda_n^{k+l}(%
	e_\bfj%
	\otimes e_\bfi) \\
	&=%
	\sum_{\substack{\bfj \in [n]^k\\ \pi \leq \ker\bfj}}%
	\sum_{\bfs \in [n]^k, \ \bfr \in [n]^l} %
	 e_\bfs %
	 \otimes e_\bfr  %
	 \otimes  pu_{s_1j_1} \dotsb u_{s_kj_k}u_{r_1i_1} \dotsb u_{r_li_l}p \\
	 &=%
	 \sum_{\bfs \in [n]^k, \ \bfr \in [n]^l} %
	 e_\bfs %
	 \otimes e_\bfr  %
	 \otimes%
	 \sum_{\substack{\bfj \in [n]^k\\ \pi \leq \ker\bfj}}%
	 (pu_{s_1j_1} \dotsb u_{s_kj_k})u_{r_1i_1} \dotsb u_{r_li_l}p.
	\end{align*}
	By  Lemma~\ref{redefine bool_eqg},
	we have 
	$\sum_{\bfj \in [n]^k, \ \pi \leq \ker\bfj}%
	pu_{s_1j_1} \dotsb u_{s_kj_k}%
	= \zeta(\pi,\ker \bfs) p.
	$
	Hence
	\begin{align*}
		\Lambda_n^{k+l}(T_\pi \otimes e_\bfi) %
		= %
		\sum_{\substack{\bfs \in [n]^k\\ \pi \leq \ker\bfs}}
		\sum_{\bfr \in [n]^l} %
		e_\bfs %
		\otimes e_\bfr  %
		\otimes%
		pu_{r_1i_1} \dotsb u_{r_li_l}p
		=%
		T_\pi \otimes \Lambda_n^l(e_\bfi).
	\end{align*}
The proof for the second equation is similar to that of the first one.

\end{proof}

\begin{lem}\label{Haar_lemma}
Let $D$ be a category of interval partitions with  $D(l) \neq \emptyset$ for a fixed index $l \in \N$.
For any $k \in L_D$ and $\bfj \in [n]^l$, we have
	\begin{align*}
	H^{D(k+l)}(T_{{\bf1}_k} \otimes e_\bfj) = T_{{\bf1}_k} \otimes H^{D(l)}e_\bfj.
	\end{align*}

\end{lem}
\begin{proof}

	Since $D$ is $\otimes$-stable,
	we have
	$H^{D(k+l)} \geq  H^{D(k)} \otimes H^{D(l)}$.
	As $D(k), D(l) \neq \emptyset$, it holds that $D(k+l) \neq \emptyset$.
	We have $H^{D(k+l)}(T_{{\bf1}_k} \otimes H^{D(l)}e_\bfj) = T_{{\bf1}_k} \otimes H^{D(l)}e_\bfj$. We only need to show that
	\begin{align}\label{Haar_claim1_ISTS}
	\langle T_\pi , T_{{\bf1}_k} \otimes e_\bfj \rangle = %
	\langle T_\pi, T_{{\bf1}_k} \otimes H^{D(l)}e_\bfj \rangle, %
	\ \text{for any } \pi \in D(k+l).
	\end{align}

	 As $D(l) \neq \emptyset$ by the assumption, there are scalars $(\alpha_\sigma)_{\sigma \in D(l)}$ with $H^{D(l)}e_\bfj = \sum_{\sigma \in D(l)}\alpha_\sigma T_\sigma$. Then for any $\rho \in D(l)$,
	\begin{align}\label{Haar_claim1_r}
	\langle T_\rho, e_\bfj \rangle=%
	\sum_{\sigma \in D(k)}\alpha_\sigma \langle T_\rho, T_\sigma \rangle =  \sum_{\sigma \in D(l)} \alpha_\sigma n^{|\rho \vee \sigma  |}.
	\end{align}
	For any $\pi \in D(k+l)$,
	\begin{align}\label{Haar_claim1_l}
	\langle T_\pi, T_{{\bf1}_k} \otimes H^{D(l)}e_\bfj \rangle%
	=\sum_{\sigma \in D(l)} %
	\alpha_\sigma %
	\langle  T_\pi, T_{{\bf1}_k} \otimes T_\sigma \rangle %
	= \sum_{\sigma \in D(l)} %
	\alpha_\sigma n^{|\pi \vee ({\bf1}_k \otimes \sigma)  |}.
	\end{align}

	Consider the case $k \sim^\pi k+1$.
	Set $\pi':= \pi |_{[k+1,k+l]} $.
	We have $\pi \vee ({\bf1}_k \otimes \sigma)  = (\singleton^{\otimes (k-1)}\otimes \sqcap \otimes \singleton^{\otimes l-1}) \vee ({\bf1}_k \otimes (\pi' \vee \sigma))$.
	Hence $|\pi \vee ({\bf1}_k \otimes \sigma)  |=|\pi' \vee \sigma|$.
	By \eqref{Haar_claim1_r}, \eqref{Haar_claim1_l},
	\[
	\langle T_\pi, T_{{\bf1}_k} \otimes H^{D(l)}e_\bfj \rangle%
	=\sum_{\sigma \in D(l)} \alpha_\sigma n^{|\pi' \vee \sigma|}%
	=\langle T_{\pi'}, e_\bfj  \rangle.
	\]
	As $k \sim^\pi k+1$, we have 
	$
	\langle T_\pi , T_{{\bf1}_k} \otimes e_\bfj \rangle%
	= \langle T_\pi, e_{j_1}^{\otimes k} \otimes e_\bfj \rangle%
	= \langle T_{\pi'}, e_\bfj  \rangle.
	$
	Hence in this case we have shown \eqref{Haar_claim1_ISTS}.

	Consider the case $k \not \sim^\pi k+1$. Since $D$ is block-stable,  there are $\pi_1 \in D(k)$ and $\pi_2 \in D(l)$ with $\pi = \pi_1 \otimes \pi_2$.
	Then $\pi \vee ({\bf1}_k \otimes \sigma)  ={\bf1}_k  \otimes (\pi_2 \vee \sigma   )$, and $|\pi \vee ({\bf1}_k \otimes \sigma)  | = 1+|\pi_2 \vee \sigma   |.$
	By \eqref{Haar_claim1_r}, \eqref{Haar_claim1_l},
	\[
	\langle T_\pi, T_{{\bf1}_k} \otimes H^{D(l)}e_\bfj \rangle%
	= \sum_{\sigma \in D(l)} \alpha_\sigma n^{1+| \pi_2 \vee \sigma|} %
	=n \langle T_{\pi_2}, e_\bfj  \rangle%
	= \langle T_{\pi_1}, T_{{\bf1}_k}  \rangle \langle  T_{\pi_2}, e_\bfj \rangle
	=\langle T_\pi , T_{{\bf1}_k} \otimes e_\bfj \rangle.
	\]
	Hence we have shown \eqref{Haar_claim1_ISTS}. Then we have proven the lemma.
\end{proof}

\begin{thm}[The Haar Functionals]
Assume $D$ is blockwise. Then for any $k \in \N$,
	\begin{align}\label{fix}
	\mathrm{Fix}(\Lambda_n^k) = \mathrm{Span}\{T_\pi : \pi \in D(k) \}.
	\end{align}
Moreover, for any $n \in \N$ there exists the unique unital Haar functional $h_D$ on $\mathcal{A}_p[D; n]$ with
\begin{enumerate}
\item \label{Haar_unit}
	$h_D(p) = 1,$
\item \label{Haar_def}
	$h_D(pu_{\bfi \bfj} p) = H^{D(k)}_{\bfi \bfj}$ for $\bfi, \bfj \in [n]^k$, and $k \in \N$,
\item \label{Haar_prod}
	$h_D(a_1 \cdots a_l)  = h_D(a_1) \cdots h_D(a_l)$ for any $l \in \N$, $a_1 , \dots, a_l \in V_n^D$.
\end{enumerate}
\end{thm}
\begin{proof}
	By the direct calculation,
	$
		\mathrm{Fix}(\Lambda_n^k) \supset \mathrm{Span}\{T_\pi : \pi \in D(k) \}.
	$
	We prove the opposite inclusion.
	We have $
	\sum_{\bfs \in [n]^k,\ \pi \leq \ker \bfs}%
	H^{D(k)}_{\bfr \bfs}%
	= \langle e_\bfr, H^{D(k)}T_\pi  \rangle
	=\langle e_\bfr, T_\pi  \rangle
	= \zeta(\pi, \ker \bfr)$.
	 Similarly we have $
\sum_{\bfr \in [n]^k, \, \pi \leq \ker \bfr}%
H^{D(k)}_{\bfr \bfs}=\zeta(\pi, \ker \bfs).
$

Assume $k \in L_D$.
We prove that for any $l \in \N$ and $\bfi, \bfj, \bfr \in [n]^{k+l}$,
\begin{align}\label{bilinear}
	\sum_{\substack{\bfs \in [n]^k \\ 1_k \leq \ker \bfs}}%
	H^{D(k+l)}_{\bfr \amalg \bfi, \bfs \amalg \bfj} %
	= \zeta (\pi, \ker \bfr)  H^{D(l)}_{\bfi \bfj}.
	\end{align}
In the case $H^{D(k+l)} =0$ it holds that $H^{D(l)}=0$ as $D$ is $D(k) \neq \emptyset$ and (D1). Assume that $H^{D(k+l)} \neq 0$.
By  condition (D3), $D(l) \neq \emptyset$.
Thus by Lemma~ \ref{Haar_lemma}, we have
\begin{align*}
	\sum_{\substack{\bfs \in [n]^k \\ 1_k \leq \ker \bfs}}%
	H^{D(k+l)}_{\bfr \amalg \bfi, \bfs \amalg \bfj} %
	= \langle e_\bfr \otimes e_\bfi, H^{D(k+l)}(T_{1_k} \otimes e_\bfj) \rangle %
	=  \langle e_\bfr \otimes e_\bfi, T_{1_k} \otimes H^{D(l)}e_\bfj \rangle
\end{align*}
This proves the claim~ \eqref{bilinear}.
Similarly we have
\begin{align*}
	\sum_{\substack{\bfr \in [n]^k \\ \pi \leq \ker \bfr}}%
	H^{D(k+l)}_{\bfr \amalg \bfi, \bfs \amalg \bfj}
	&=\zeta (\pi, \ker \bfs)  H^{D(l)}_{\bfi \bfj}, %
	\mathrm{ \ for \ any \ } \bfs \in [n]^k, \\
	\sum_{\substack{\bfs \in [n]^k \\ \pi \leq \ker \bfs}}%
	H^{D(k+l)}_{\bfi  \amalg \bfr , \bfj \amalg \bfs}
	&=\zeta (\pi, \ker \bfr)  H^{D(l)}_{\bfi \bfj},%
	\mathrm{ \ for \ any \ } \bfr \in [n]^k,\\
	\sum_{\substack{\bfr \in [n]^k \\ \pi \leq \ker \bfr}}%
	H^{D(k+l)}_{\bfi  \amalg \bfr , \bfj \amalg \bfs}
	&=\zeta (\pi, \ker \bfs)  H^{D(l)}_{\bfi \bfj},%
	\mathrm{ \ for \ any \ } \bfs \in [n]^k.
\end{align*}
Therefore, there is a  functional $h_D$ on $V^D_n$ with \eqref{Haar_unit} and \eqref{Haar_def}.

For any $\xi \in \mathrm{Fix}(\Lambda_n^k)$,
$
 (id \otimes h_D)\Lambda_n^k(\xi)=(id \otimes h_D)(\xi \otimes p)=\xi \otimes 1.
$
On the other hand, we have 
\[
(id \otimes h_D)\Lambda_n^k(\xi)= %
\sum_{\bfi, \ \bfj \in [n]^k}%
\xi_\bfi e_\bfj \otimes H^{D(k)}_{\bfi \bfj}p%
= H^{D(k)}\xi \otimes p.
\]
Thus we have $H^{D(k)}\xi = \xi$, which proves $\mathrm{Fix}(\Lambda_n^k) = \mathrm{Span}\{T_\pi : \pi \in D(k) \}$.
As $\Lambda^k_n(H^{D(k)}e_\bfj)=H^{D(k)}e_\bfj$ it holds that $\sum_{\bfs \in [n]^k} pu_{\bfi \bfs}p H^{D(k)}_{\bfs \bfj}=H^{D(k)}_{\bfi \bfj}p$. Hence
\begin{align*}
	(id \otimes h_D) \Delta (pu_{i_1j_1} \dotsb u_{i_kj_k}p)
	&=\sum_{\bfs \in [n]^k} pu_{s_1j_1} \dotsb u_{s_kj_k}p H^{D(k)}_{\bfs \bfj}\\
	&= H^{D(k)}_{\bfi \bfj}p%
	= h_D(pu_{i_1j_1} \dotsb u_{i_kj_k}p)p.
\end{align*}
Therefore, we have $(id \otimes h_D) \Delta = h_D(\cdot)p$. The other invariance property follows from a similar proof.
By Lemma~\ref{Haar inv under prod}, we can extend $h_D$ to $\mathcal{A}_p[D;n]$ by  \eqref{Haar_prod} with the Haar invariance.

\end{proof}

\subsection{Haar states on Beq$_x$}\label{Haar states on Beq_x}

In this section, we construct a $*$-representation of $\mathcal{A}[I;n]$ on $L^2(S_n)$, which is the GNS-representation of the Haar functional $h_I$.
In particular, we see that $h_I$ is a state.

By a similar discussion, we show that $Beq_h, Beq_o$ have the unique Haar state.

\begin{note}\label{note pi s}

Let $(L^2(S_n))_{n \in \N}$ be the sequence of the Hilbert spaces of all $L^2$-functions on permutation groups $S_n$ with respect to the normalized counting measure.
Let us define  orthogonal projections $\hat{P_{ij}} \in L^2(S_n)$ $(i,j \leq n)$  and the unit vector $\hat{1} \in L^2(S_n)$ by
\[
\hat{P_{ij}}(\sigma) := \delta_{i,\sigma(j)}, \ \hat{1}(\sigma) := 1  \  (\sigma \in S_n).
\]
For $\xi \in L^2(S_n)$, let us denote by $ Q( \xi ) $
the orthogonal projection onto the one dimensional subspace $\C \xi \subset L^2(S_n)$.
We denote by $\omega$ the vector state on $B(L^2(S_n))$ induced by the unit vector $\hat{1}$.

\end{note}

We show that the operators $\hat{P_{ij}}$  and $\hat{1}$ satisfies the relations which appear in the definition of Liu's Boolean quantum permutation
semigroups $(\mc{B}_s(n))_{n \in \N}$.

\begin{prop}\label{repsentation pi_s relation}
	Let $u_{ij} (i,j \leq n)$, $p$ be the generators of $\mathcal{A}[I;n]$.
Then we have
	\begin{align*}
	Q(\hat{P_{ij_1}})Q(\hat{P_{ij_2}}) &=  \delta_{j_1, j_2}Q(\hat{P_{ij_1}}),%
	 \text{ for any } i \in [n], j_1, j_2 \in [n],\\
	Q(\hat{P_{i_1j}})Q(\hat{P_{i_2j}}) &= \delta_{i_1, i_2}Q(\hat{P_{i_1 j}}),%
	 \text{ for any } j \in [n], i_1, i_2 \in [n] ,\\
	Q(\hat{P_{ij}})Q(\hat{1}) &=  | \hat{P_{ij}} \rangle \langle \hat{1} |,%
		 \text{ for any } i,j \in [n].
	\end{align*}

\end{prop}

\begin{proof}
 For any indices $i \leq n$ and $j_1, j_2 \leq n$, it holds that
 \[
 \langle \hat{P_{ij_1}}, \hat{P_{ij_2}} \rangle = %
  \frac{\sum_{\sigma \in S_n} \delta_{\sigma(i),j_1}\delta_{\sigma(i),j_2}}{ \# S_n} = \frac{\delta_{j_1, j_2}\# S_{n-1}}{\# S_n},
 \]
 \[
  Q(\hat{P_{ij_1}})Q(\hat{P_{ij_2}}) = \frac{ \langle \hat{P_{ij_1}}, \hat{P_{ij_2}} \rangle}{ \langle \hat{P_{ij_2}}, \hat{P_{ij_2}} \rangle} Q(\hat{P_{ij_1}}) = \delta_{j_1, j_2}Q(\hat{P_{ij_1}}).
 \]
 Similarly,  $Q(\hat{P_{i_1j}})Q(\hat{P_{i_2j}}) = \delta_{i_1, i_2}Q(\hat{P_{i_1 j}})$ for any indices $i_1, i_2 \leq n$
 and $j \leq n$.
 Then
\[
	Q(\hat{P_{ij}})Q(\hat{1}) %
	= \frac{\langle \hat{P_{ij}}, \hat{1} \rangle }{ \langle \hat{P_{ij}}, \hat{P_{ij}} \rangle \langle \hat{1}, \hat{1} \rangle}| \hat{P_{ij}} \rangle \langle \hat{1} | %
	=  | \hat{P_{ij}} \rangle \langle \hat{1} |.
\]

\end{proof}

\begin{cor}
	There is the unique $*$-representation  $\pi_s : {\mathcal A}_p[I;n] \ra B(L^2(S_n))$ with
	\begin{align}\label{representation pi_s}
	\pi_s (u_{ij}) := Q(\hat{P_{ij}}) \ (i, j \leq n), \ \pi_s(p):= Q(\hat{1}).
	\end{align}
	Moreover, there are $*$-representations $\bar{\pi}_s \colon Beq_s(n) \ra B(L^2(S_n))$ and $\Pi_s \colon \mc{B}_s(n) \ra B(L^2(S_n))$ with
 $
 \bar{\pi}_s([u_{ij}]) = \pi_s(u_{ij}) = \Pi_s(U_{ij})  (i, j \leq n)$   and  %
 $\bar{\pi}_s([p]) = \pi_s(p) = \Pi_s(\mathbf{P}).
 $
\end{cor}

\begin{proof}
Since $\sum_{i=1}^{n}\hat{P_{ij}} =   \hat{1} \ (j \leq n)$ and $\sum_{j=1}^{n}\hat{P_{ij}} =  \hat{1} \ (i \leq n)$,
 we have for any $k \in \N$ and $\bfi, \bfj \in [n]^k$,
	\begin{align*}
	\sum_{i =1}^n Q(\hat{P_{ij_1}}) \dots Q(\hat{P_{ij_k}})Q(\hat{1}) %
	&= \begin{cases} Q(\hat{1}), & j_1=\dotsb=j_k, \\ 0, & \text{otherwise,}\end{cases} \\
	\sum_{j =1}^n Q(\hat{P_{i_1 j}}) \dots Q(\hat{P_{i_k j}})Q(\hat{1}) %
	&= \begin{cases} Q(\hat{1}), & i_1=\dotsb=i_k, \\ 0, & \text{otherwise.}\end{cases} \\
	\end{align*}
Hence the $*$-representation $\pi_s$  \eqref{representation pi_s} is well-defined.
The existence of $\Pi_s$  directory follows from Proposition~\ref{repsentation pi_s relation}.
Since $|| \pi_s(u_{ij}) ||_n \leq 1  $ and $\pi_s(p) =1$, we have $\bar{\pi}_s$ is well-defined.
\end{proof}

\begin{lem}
For any $\bfi, \bfj \in [n]^k$ and $k \in \N$, we have
\begin{align}\label{explicit_Haar_s}
	Q(\hat{1})Q(\hat{P_{i_1 j_1 }}) \dots Q(\hat{P_{i_k j_k }})Q(\hat{1})%
	=  \frac{\delta(\inf_I \ker\bfi, \inf_I \ker\bfj ) }{n (n-1)^{| \inf_I \ker\bfi | -1}} Q(\hat{1}).
\end{align}

\end{lem}
\begin{proof}
In the case $| \inf_I \ker\bfi | = 1$ it holds that $\ker \bfi =  {\bf 1}_k$.
Then the left hand side of the equation \eqref{explicit_Haar_s} is equal to
\[
\delta({\bf1}_k, \ker \bfj ) Q(\hat{1})Q(\hat{P_{i_1 j_1 }}) Q(\hat{1})  = \frac{\delta({\bf1}_k, \ker \bfj ) }{n}   Q(\hat{1}).
\]
This proves  \eqref{explicit_Haar_s}.

Let $m \in \N$ and assume that the equation holds if $| \inf_I \ker\bfi | \leq m$.
Let $\inf_I \ker \bfi = \{V_1 < V_2 < \dots < V_b\},$ where $b = | \inf_I \ker\bfi |$, and $s_\nu := \min V_\nu$ for $\nu \in [b]$.
Then the left hand side of the equation \eqref{explicit_Haar_s} is equal to
\[
\zeta(\inf_I \ker \bfi, \ker \bfj) %
Q(\hat{1})Q(\hat{P_{i(s_1) j(s_1) }}) \cdots Q(\hat{P_{i(s_b)j(s_b)}})Q(\hat{1}).
\]
Since $i(s_\nu) \neq i(s_{\nu + 1})$ for any $\nu < k$,
$
Q(\hat{P_{i(s_\nu) j(s_\nu) }})Q(\hat{P_{i(s_{\nu + 1}) j(s_{\nu + 1}) }}) =  0,
$
whenever  $j(s_\nu) = j(s_{\nu + 1})$.
Now \[
\zeta(\inf_I \ker \bfi, \ker \bfj) \prod_{\nu = 1}^{b}1(j(s_\nu) \neq j(s_{\nu + 1})) = \delta(\inf_I \ker \bfi, \inf_I \ker \bfj).
\]
Assume indices satisfy $i_1 \neq i_2$ and $j_1 \neq j_2$. Then
\[
\langle \hat{P_{i_1 j_1}}, \hat{P_{i_2 j_2}} \rangle = %
\frac{\sum_{\sigma \in S_n} \delta_{\sigma(i_1),j_1}\delta_{\sigma(i_2),j_2}}{ \# S_n} = \frac{\# S_{n-2}}{\#S_n}.
\]
Hence, if $\inf_I \ker \bfi =  \inf_I \ker \bfj$, we have
\begin{align*}
Q(\hat{1})Q(\hat{P_{i(s_1) j(s_1) }}) \dots Q(\hat{P_{i(s_b)j(s_b)}})Q(\hat{1})%
&= \frac{%
	(\#S_{n-1}/ \#S_n)^2%
	(\#S_{n-2} / \#S_n)^{b-1}%
	}%
	{%
		(\#S_{n-1} / \#S_n)^b%
		}%
		Q(\hat{1})\\
&=	\frac{%
	(\#S_{n-2})^{b-1}%
	}%
	{%
	\# S_n%
	(\# S_{n-1})^{b-2}%
		}Q(\hat{1})
		= \frac{1}{n(n-1)^{b-1}}		Q(\hat{1}).
\end{align*}
It proves the lemma.
\end{proof}

\begin{lem}
	Let $u_{ij} (i, j \in [n]) $ and $p$ be the generaters of ${\mathcal A}[I;n]$.
  Then for any $k \in \N$, $\pi \in I(k)$ and $\bfi, \bfj \in [n]^k$,
  \begin{align}
    \sum_{\substack{\bfr \in [n]^k, \\ \inf_I \ker \bfr = \pi }} %
    u_{\bfr \bfj} p %
    &= \delta( \pi, \inf_I \ker_I \bfj)p. \label{relation ker s row} \\
    \sum_{\substack{\bfs \in [n]^k, \\ \inf_I \ker \bfs = \pi }} %
    u_{\bfi \bfs} p %
    &= \delta(\inf_I \ker \bfi, \pi)p, \label{relation ker s col}
\end{align}
\end{lem}

\begin{proof}
  We give the proof only for the equation \eqref{relation ker s row};
  the same proof runs for the other.
  The proof is by induction on $| \pi |$.
  In the case $| \pi| = 1$, we have $\pi = {\bf 1}_k$.
  Then for any $\bfr \in [n]^k$, it holds that $\inf_I \ker \bfr = \pi$ if and only if
  $\ker \bfr = {\bf 1}_k$. This gives the  equation \eqref{relation ker s row}.

  Let  $b \in \N$. Assume \eqref{relation ker s row} holds in the case $|\pi| = b$.
  In the case $|\pi| = b + 1$, write $\pi = \{V_1 <  V_2 < \dots < V_{b+1}\}$.
  Set $v = \max V_b$.

  \setlength{\unitlength}{0.5cm}
  \begin{center}
  \begin{picture}(18,3)
  \put(-2, 2){$\pi=$}
  \put(-0.1,2){\uppartiii{1}{1}{2}{4}}
  \put(2.75,2.3){$\ldots$}
  \put(2.5,0.75){$V_1$}
  \put(1,1.4){$1$}

  \put(4,2){\uppartiii{1}{1}{2}{4}}
  \put(6.75,2.3){$\ldots$}
  \put(6.5,0.75){$V_2$}

  \put(9,2.3){$\cdots$}
  \put(9,0.75){$\cdots$}

  \put(9.5,2){\uppartiii{1}{1}{2}{4}}
  \put(12.25,2.3){$\ldots$}
  \put(12,0.75){$V_b$}
  \put(13.5,1.4){$v$}

  \put(13.5,2){\uppartiii{1}{1}{2}{4}}
  \put(16.25,2.3){$\ldots$}
  \put(16,0.75){$V_{b+1}$}
  \put(17.5,1.4){$k$}

  \end{picture}
  \end{center}
  Then  the left hand side of \eqref{relation ker s row} is
  equal to
  \begin{align}\label{temp_1}
    \sum_{\substack{\bfr \in [n]^{[k] \setminus V_{b+1}}, \\ \inf_I \ker \bfr = \pi|_{[k] \setminus V_{b+1}} }} %
     ( u_{r_1 j_1} u_{r_2 j_2} \dots u_{r_v j_v}%
     \sum_{\substack{r' \in [n],\\ r' \neq r_v }} %
     u_{r' j_{v+1}}u_{r' j_{v+2}} \dots u_{r' j_{k}}p).
  \end{align}
  It follows that
  \[
  \sum_{\substack{r' \in [n],\\ r' \neq r_v }} %
  u_{r' j_{v+1}}u_{r' j_{v+2}} \dots u_{r' j_{k}}p%
  = \delta (\ker (\bfj|_{V_{b+1}}) , {\bf 1} _{V_{b+1}} ) p %
  -   u_{r_v j_{v+1}}u_{r_v j_{v+2}} \dots u_{r_v j_{k}}p.
  \]
  Then by the assumption of induction, \eqref{temp_1} is equal to
  \begin{align}\label{temp_2}
    \delta(\pi|_{[k] \setminus V_{b+1}},& \inf_I \ker (\bfj|_{[k] \setminus V_{b+1}}))%
    p%
    \cdot \delta \big( {\bf 1}_{V_{b+1}}, \ker (\bfj |_{V_{b+1}}) \big) - R,
  \end{align}
  where
  \begin{align*}
    R = %
    \sum_{\substack{\bfr \in [n]^{[k] \setminus V_{b+1}}, \\ \inf_I \ker \bfr = \pi|_{[k] \setminus V_{b+1}} }} %
     u_{r_1 j_1} u_{r_2 j_2} \dots u_{r_v j_v}%
     \cdot u_{r_v j_{v+1}}u_{r_v j_{v+2}} \dots u_{r_v j_{k}}p.
  \end{align*}
  For any multi-index $\bfr \in [n]^{[k] \setminus V_{b+1}}$,
  set $\tilde{\bfr} \in [n]^k$ by  $\tilde{\bfr}_m := \bfr_m$ if $m \leq v$, and $\tilde{\bfr}_m := r_v$, otherwise.
  Then $\inf_I \ker \bfr = \pi|_{[k] \setminus V_{b+1}}$ if and only if $\inf_I \ker \tilde{\bfr} = \tilde{\pi}$,
  where $\tilde{\pi} :=   \pi \vee (\singleton^{\otimes (v-1)} \otimes \sqcap \otimes \singleton^{\otimes (k-v-1)}  )$.
  We see that the partition $\tilde{\pi} $ is drown as the following figure.
\setlength{\unitlength}{0.5cm}
\begin{center}
\begin{picture}(18,3)
\put(-2,2){$\tilde{\pi}=$}
\put(-0.1,2){\uppartiii{1}{1}{2}{4}}
\put(2.75,2.3){$\ldots$}
\put(2.5,0.75){$V_1$}
\put(1,1.4){$1$}

\put(4,2){\uppartiii{1}{1}{2}{4}}
\put(6.75,2.3){$\ldots$}
\put(6.5,0.75){$V_2$}

\put(9,2.3){$\cdots$}
\put(9,0.75){$\cdots$}

\put(9.5,2){\uppartiii{1}{1}{2}{4}}
\put(12.25,2.3){$\ldots$}
\put(13.5,1.4){$v$}

\put(12.5,2){\uppartii{1}{1}{2}}
\put(13.5,0.75){$V_b \cup V_{b+1}$}

\put(13.5,2){\uppartiii{1}{1}{2}{4}}
\put(16.25,2.3){$\ldots$}
\put(17.5,1.4){$k$}

\end{picture}
\end{center}
Since $|\tilde{\pi}| = b$, applying the assumption of induction yields
$
R = \delta(\tilde{\pi}, \inf_I \ker \bfj) p.
$
Hence \eqref{temp_2} is equal to
\begin{align*}
\big[\delta &\big(\pi|_{[k]  \setminus V_{b+1}}, \inf_I \ker (\bfj|_{[k] \setminus V_{b+1}})\big)%
\cdot \delta({\bf 1}_{V_{b+1}}, \ker (\bfj |_{V_{b+1}}) )%
- %
\delta(\tilde{\pi}, \inf_I \ker \bfj)\big]
p
= \delta(\pi, \inf_I \ker \bfj)p.
\end{align*}
This is the desired conclusion.
\end{proof}

\begin{prop}\label{prop extension of Haar}
	The functional $h_I$ is a Haar state
	 and the triplet $(\pi_s, L^2(S_n), \hat{1})$ is the GNS-representation of the pair $(\mathcal{A}_p[I;n], h_I)$.

\end{prop}
\begin{proof}
  Our proof starts with the observation that  the functional $\omega \circ \pi_s$ satisfies the Haar invariance on $V_n^s$.
	For any $k \in \N$, $\bfi, \bfj \in [n]^k$,
  \begin{align*}
	 (id \otimes \omega \circ \pi_s)\Delta(pu_{\bfi \bfj}p) &= %
	 \sum_{\bfs \in [n]^k} %
	 pu_{\bfi \bfs} p  \cdot %
  	 \frac{\delta(\inf_I \ker\bfs, \inf_I \ker\bfj ) }{n (n-1)^{| \inf_I \ker\bfs | -1}} \\
	 &=  	 \frac{1 }{n (n-1)^{| \inf_I \ker\bfj | -1}}%
	 \sum_{\substack{\bfs \in [n]^k, \\ \inf_I \ker \bfs = \inf_I \ker \bfj }} %
	 pu_{\bfi \bfs} p.  %
	\end{align*}
  By the equation \eqref{relation ker s row} , we have for any interval partition $\pi \in I(k)$,
	\[
		 \sum_{\substack{\bfs \in [n]^k, \\ \inf_I \ker \bfs = \pi }} %
		 pu_{\bfi \bfs} p %
		 = \delta(\inf_I \ker \bfi, \pi)p.
	\]
  From this, we obtain the half of the Haar invariance of $\omega \circ \pi_s$.
  Similar arguments can be applied to the other invariance.
  By the uniqueness of the Haar functional (Lemma~\ref{uniqueness Haar}) , we have proven the proposition.

\end{proof}

\begin{thm}\label{Haar state h_s existence}
  For any $n \in \N$,  $Beq_s(n)$ and $\mc{B}_s(n)$ admit the unique Haar states.
  We write them $h_s$ and $h_{\mc{B}_s}$, respectively.  Furthermore, we have $h_{\mc{B}_s} \circ \alpha = h_s$.
\end{thm}
\begin{proof}
  The existence of a Haar state follows immediately from Proposition \ref{prop extension of Haar}. The uniqueness follows from Proposition~\ref{uniqueness Haar}.
\end{proof}

\begin{lem}
 Assume the index $x$ be $o$ or $h$.
	Let $u_{ij} (i, j \in [n]) $ and $p$ be the generaters of ${\mathcal A}[I_x;n]$.
  	Then for any $k \in \N$, $\pi \in I_x(2k)$ and multi-indices  $\bfi, \bfj \in [n]^{2k}$ with
  	$\sqcap^{\otimes k} \leq \ker\bfi, \ker\bfj$,
  	it holds that
  \begin{align}
    \sum_{\substack{\bfr \in [n]^{2k}, \\ \inf_{I_x} \ker \bfr = \pi }} %
    u_{\bfr \bfj} p %
    &= \delta( \pi, \inf_{I_x} \ker_I \bfj)p, \label{relation ker x row} \\
    \sum_{\substack{\bfs \in [n]^{2k}, \\ \inf_{I_x} \ker \bfs = \pi }} %
    u_{\bfi \bfs} p %
    &= \delta(\inf_{I_x} \ker \bfi, \pi)p. \label{relation ker x col}
	\end{align}
\end{lem}
\begin{proof}
  We only prove the first equation.
 In the case of $x=o$,
 we have $\pi = \sqcap^{ \otimes k}$ and $\inf_{I_o}\ker \rho = \sqcap^{ \otimes k}$ for any $\rho \in P(2k)$ with $\rho \geq \sqcap^{ \otimes k}$. Hence the first equation follows from the definiton.

 In the case of $x=h$,
 the proof is by induction on $| \pi |$.
  In the case $| \pi| = 1$, we have $\pi = {\bf 1}_{2k}$.
  Then for any $\bfr \in [n]^{2k}$, it holds that $\inf_{I_h} \ker \bfr = \pi$ if and only if
  $\ker \bfr = {\bf 1}_{2k}$. This gives  \eqref{relation ker x row}.

  Let  $b \in \N$. Assume the first equation holds in the case $|\pi| = b$.
  In the case $|\pi| = b + 1$, write $\pi = \{V_1 <  V_2 < \dots < V_{b+1}\}$.
  Set $v = \max V_b$.

  \setlength{\unitlength}{0.5cm}
  \begin{center}
  \begin{picture}(18,3)
	\put(-2,2){$\pi=$}
  \put(0.4,2){\uppartiii{1}{1}{2}{4}}
  \put(4.5,2){\uppartiii{1}{1}{2}{4}}
  \put(10,2){\uppartiii{1}{1}{2}{4}}
  \put(14,2){\uppartiii{1}{1}{2}{4}}

  \put(-0.1,2){\uppartiii{1}{1}{2}{4}}

  \put(2.75,2.3){$\ldots$}
  \put(2.5,0.75){$V_1$}
  \put(1,1.4){$1$}
  \put(1.5,1.4){$2$}

  \put(4,2){\uppartiii{1}{1}{2}{4}}

  \put(6.75,2.3){$\ldots$}
  \put(6.5,0.75){$V_2$}

  \put(9,2.3){$\cdots$}
  \put(9,0.75){$\cdots$}

  \put(9.5,2){\uppartiii{1}{1}{2}{4}}

  \put(12.25,2.3){$\ldots$}
  \put(12,0.75){$V_b$}
  \put(14,1.4){$v$}

  \put(13.5,2){\uppartiii{1}{1}{2}{4}}
  \put(16.25,2.3){$\ldots$}
  \put(16,0.75){$V_{b+1}$}
  \put(18,1.4){$2k$}

  \end{picture}
  \end{center}
%
  Then  the left hand side of \eqref{relation ker x row} is
  equal to
  \begin{align*}
    \sum_{\substack{\bfr \in [n]^{[2k] \setminus V_{b+1}}, \\ \inf_{I_h} \ker \bfr = \pi|_{[k] \setminus V_{b+1}} }} %
     ( u_{r_1 j_1} u_{r_2 j_2} \dots u_{r_v j_v}%
     \sum_{\substack{r' \in [n],\\ r' \neq r_v }} %
     u_{r' j_{v+1}}u_{r' j_{v+2}} \dots u_{r' j_{2k}}p). \label{temp_1 x}
  \end{align*}
 Since $|V_{b+1}|$ is even,  it follows that
  \[
  \sum_{\substack{r' \in [n],\\ r' \neq r_v }} %
  u_{r' j_{v+1}}u_{r' j_{v+2}} \dots u_{r' j_{2k}}p%
  = \delta (\ker (\bfj|_{V_{b+1}}) , {\bf 1} _{V_{b+1}} ) p %
  -   u_{r_v j_{v+1}}u_{r_v j_{v+2}} \dots u_{r_v j_{2k}}p.
  \]
  By the assumption of induction, \eqref{relation ker x row} is equal to
  \begin{align}\label{temp_2 x}
    \delta(\pi|_{[k] \setminus V_{b+1}},& \inf_{I_h} \ker (\bfj|_{[k] \setminus V_{b+1}}))%
    p%
    \cdot \delta \big( {\bf 1}_{V_{b+1}}, \ker (\bfj |_{V_{b+1}}) \big) - R,
  \end{align}
  where
  \begin{align*}
    R = %
    \sum_{\substack{\bfr \in [n]^{[2k] \setminus V_{b+1}}, \\ \inf_{I_h} \ker \bfr = \pi|_{[2k] \setminus V_{b+1}} }} %
     u_{r_1 j_1} u_{r_2 j_2} \dots u_{r_v j_v}%
     \cdot u_{r_v j_{v+1}}u_{r_v j_{v+2}} \dots u_{r_v j_{2k}}p.
  \end{align*}
  For any multi-index $\bfr \in [n]^{[2k] \setminus V_{b+1}}$,
  set $\tilde{\bfr} \in [n]^{2k}$ by  $\tilde{\bfr}_m := \bfr_m$ if $m \leq v$, $\tilde{\bfr}_m := r_v$, otherwise.

  Set  $\tilde{\pi} :=   \pi \vee (\singleton^{\otimes (v-1)} \otimes \sqcap \otimes \singleton^{\otimes (2k-v-1)}  )$.
    The partition $\tilde{\pi} $ can be  drown as the following figure.
    \setlength{\unitlength}{0.5cm}
    \begin{center}
    \begin{picture}(18,3)
  		\put(-2,2){$\tilde{\pi}=$}

    \put(0.4,2){\uppartiii{1}{1}{2}{4}}
    \put(4.5,2){\uppartiii{1}{1}{2}{4}}
    \put(10,2){\uppartiii{1}{1}{2}{4}}
    \put(14,2){\uppartiii{1}{1}{2}{4}}

  	\put(13,2){\uppartii{1}{1}{2}}
    \put(14,0.8){$\cup$}

    \put(-0.1,2){\uppartiii{1}{1}{2}{4}}

    \put(2.75,2.3){$\ldots$}
    \put(2.5,0.75){$V_1$}
    \put(1,1.4){$1$}
    \put(1.5,1.4){$2$}

    \put(4,2){\uppartiii{1}{1}{2}{4}}

    \put(6.75,2.3){$\ldots$}
    \put(6.5,0.75){$V_2$}

    \put(9,2.3){$\cdots$}
    \put(9,0.75){$\cdots$}

    \put(9.5,2){\uppartiii{1}{1}{2}{4}}

    \put(12.25,2.3){$\ldots$}
    \put(12,0.75){$V_b$}
    \put(14,1.4){$v$}

    \put(13.5,2){\uppartiii{1}{1}{2}{4}}
    \put(16.25,2.3){$\ldots$}
    \put(16,0.75){$V_{b+1}$}
    \put(18,1.4){$2k$}

    \end{picture}
    \end{center}
  Then $\inf_{I_h} \ker \bfr = \pi|_{[2k] \setminus V_{b+1}}$ if and only if $\inf_{I_h} \ker \tilde{\bfr} = \tilde{\pi}$.
Since $|\tilde{\pi}| = b$, applying the assumption of induction yields
$
R = \delta(\tilde{\pi}, \inf_{I_h} \ker \bfj) p.
$
Hence \eqref{temp_2 x} is equal to
\begin{align*}
\big[\delta &\big(\pi|_{[k]  \setminus V_{b+1}}, \inf_{I_h} \ker (\bfj|_{[2k] \setminus V_{b+1}})\big)
\cdot \delta({\bf 1}_{V_{b+1}}, \ker (\bfj |_{V_{b+1}}) )%
- %
\delta(\tilde{\pi}, \inf_{I_h} \ker \bfj)\big]
p
= \delta(\pi, \inf_{I_h} \ker \bfj)p.
\end{align*}
This is the desired conclusion.
\end{proof}

Let us construct $*$-representations of $\mathcal{A}[I_o;n]$, $\mathcal{A}[I_h;n]$,
which give us Haar states.
	We set 	a one dimensional projection $R$ and self-adjoint operators $F_i \in M_{n+1}(\C) $ $(i \leq n)$ by the following: for $k, l \leq n+1$,
	\begin{align*}
				R(k,l)= \begin{cases}
					1, \text{ if } k=l=n+1,\\
					0, \text{ otherwise},
				\end{cases} \
		F_i(k,l)= \begin{cases}
			1, \text{ if } (k,l)=(i, n+1), (n+1,i),\\
			0, \text{ otherwise}.
		\end{cases}
	\end{align*}
	For any $i, r \in [n]$ with $i \neq r$ we have
	\begin{align}\label{relations of rep for h, o}
		R F_i^2 = R , \ R F_i F_r = 0.
	\end{align}
	Set $F_{ij} = F_i \otimes F_j$.
		 We set operators
		\begin{align*}
			 P^o &:=  R\otimes R, \ %
			U_{ij}^o := \frac{1}{\sqrt{n}}F_{ij}, \\
			P^h &:= Q(\hat{1}) \otimes  P^o, \ %
			U^h_{ij} := Q(\hat{P_{ij}}) \otimes F_{ij}.
		\end{align*}

\begin{lem}
	The following relations define a $*$-homomorphism $\pi_o \colon  {\mathcal A}[I_o;n] \ra M_{n+1}(\C)$
	and a $*$-homomorphism $\pi_h \colon {\mathcal A}[I_o;n] \ra B(L^2(S_n)) \otimes M_{n+1}(\C)$.
\begin{align*}
	\pi_x(p^x)= P^x, \ \pi_x(u^x_{ij}) = U^x_{ij}.
\end{align*}

\end{lem}

\begin{proof}
The proof is straightforward.
\end{proof}

\begin{lem}
	Let $l, n \in \N$. If $l$ is odd then for any $\bfi, \bfj \in [n]^l$, we have
	\begin{align}\label{odd vanish}
			\pi_o (p u_{\bfi\bfj}p ) =  \pi_h (p u_{\bfi\bfj}p ) = 0.
	\end{align}
	If $l$ is even and $l=2k$, then for any $\bfi, \bfj \in [n]^{2k}$, we have
\begin{align}\label{compute pi_o}
\pi_o (p u_{\bfi\bfj}p ) &= %
 \zeta(\sqcap^{\otimes k}, \ker\bfi)
\zeta(\sqcap^{\otimes k}, \ker\bfj)
 \frac{1}{n^k} %
 \cdot P_o.\\
 \pi_h (p u_{\bfi\bfj}p ) &= %
 \zeta(\sqcap^{\otimes k}, \ker\bfi)
 \zeta(\sqcap^{\otimes k}, \ker\bfj)
 \delta(\inf_{I_h}\ker\bfi, \inf_{I_h}\ker\bfj) \frac{1}{n (n-1)^{|\inf_{I_h}\ker\bfi| - 1} } \cdot P_h. \label{compute pi_h}
\end{align}

\end{lem}
\begin{proof}
	The first and the second equations  follow directory from \eqref{relations of rep for h, o}.
	We prove the last equation.
If $i \neq r$, or $j \neq s$, we have
$
P_h U^h_{ij}U^h_{rs} = 0
$
. Hence if $\zeta(\sqcap^{\otimes k}, \ker\bfi) = 0$ or $\zeta(\sqcap^{\otimes k}, \ker\bfj) =0$ then $\pi_h(pu_{\bfi \bfj}) =0$.

Assume $\zeta(\sqcap^{\otimes k}, \ker\bfi) = 1$ and $\zeta(\sqcap^{\otimes k}, \ker\bfj) =1$.
Then $\inf_I \ker \bfi = \inf_{I_h} \ker \bfi$ and
$\inf_I \ker \bfj = \inf_{I_h} \ker \bfj$.
We check that
$
	P_h {U^h_{ij}}^2 = Q(\hat{1})Q(\hat{P_{ij}})^2 \otimes R \otimes R = %
	P_h (Q(\hat{P_{ij}})^2 \otimes 1 \otimes 1)
$
. By \eqref{explicit_Haar_s},
\begin{align*}
\pi_h(pu_{\bfi \bfj}p) &= %
	Q(\hat{1})Q(\hat{P_{i_1 j_1 }})^2 \cdots  Q(\hat{P_{i_{2k-1}j_{2k-1 }}})^2 Q(\hat{1})%
	\otimes R \otimes R \\
	&= \frac{\delta(\inf_I \ker\bfi, \inf_I \ker\bfj ) }{n (n-1)^{| \inf_I \ker\bfi | -1}} Q(\hat{1}) %
	\otimes R \otimes R.\\
	& = \frac{\delta(\inf_{I_h} \ker\bfi, \inf_{I_h} \ker\bfj ) }{n (n-1)^{| \inf_{I_h} \ker\bfi | -1}} P_h.
\end{align*}
This finishes proof.
\end{proof}

\begin{note}
We define  states $\omega_o$  on $M_{n+1}(\C)$  and  $\omega_h$ on $ B(L^2(S_n)) \otimes M_{n+1}(\C)$ by
\begin{align*}
\omega_o:= \frac{tr_{n+1} (P_o \cdot)}{tr_{n+1}P_o}, \
\omega_h:= \frac{\omega \otimes tr_{n+1} (P_h \cdot)}{\omega \otimes tr_{n+1}(P_h)}.
\end{align*}

\end{note}

\begin{prop}
For $x=o,h$, each state $\omega_x \circ \pi_x$ is a Haar state.
Furthermore,
$ 
h_{I_x}= \omega_x \circ \pi_x.
$ 
\end{prop}
\begin{proof}
	If $l \in \N$ is odd,  by \eqref{odd vanish},
	$ 	(id \otimes \omega \circ \pi_x)\Delta(pu_{\bfi \bfj}p) = %
0= \omega_x(pu_{\bfi \bfj}p)$, where  $\bfi, \bfj \in [n]^l$ and $x=o, h$.

Assume $l \in \N$ is even and set $l=2k$.
By \eqref{compute pi_o} and  \eqref{compute pi_h}, we have
\begin{align*}
		(id \otimes \omega_o \circ \pi_o)\Delta(pu_{\bfi \bfj}p)&= %
		  	\sum_{\bfs \in [n]^{2k}} %
		  	pu_{\bfi \bfs} p  \cdot %
		  		\zeta(\sqcap^{\otimes k}, \ker\bfs)
		  		\zeta(\sqcap^{\otimes k}, \ker\bfj)
				\frac{1}{n^k} \\
		  	&=\frac{\zeta(\sqcap^{\otimes k}, \ker\bfj)}{n^k} %
		  	\sum_{\substack{\bfs \in [n]^k, \\ 	\sqcap^{\otimes k} \leq \ker\bfs}} %
		  	pu_{\bfi \bfs} p.  \\%
	  	(id \otimes \omega_h \circ \pi_h)\Delta(pu_{\bfi \bfj}p) &= %
	  	\sum_{\bfs \in [n]^{2k}} %
	  	pu_{\bfi \bfs} p  \cdot %
	  	 \delta(\inf_{I_h}\ker\bfs, \inf_{I_h}\ker\bfj) %
	  	 \frac{	  	 \zeta(\sqcap^{\otimes k}, \ker\bfs)
	  	 	\zeta(\sqcap^{\otimes k}, \ker\bfj)
	  	 	}{n (n-1)^{|\inf_{I_h}\ker\bfs| - 1} } \\
	  	 	&=  		 \frac{  	 \zeta(\sqcap^{\otimes k}, \ker\bfj) }{n (n-1)^{| \inf_{I_h} \ker\bfj | -1}}%
	  	\sum_{\substack{\bfs \in [n]^k, \\ \inf_{I_h} \ker \bfs = \inf_{I_h} \ker \bfj }} %
	  	pu_{\bfi \bfs} p.  %
	  \end{align*}
By \eqref{relation ker x row}, we obtain the half of the Haar invariance of $\omega_x \circ \pi_x$ $(x=o,h)$.
Similar arguments can be applied to the other invariance.
By the uniqueness of the Haar functional (Lemma~\ref{uniqueness Haar}), we have proven the proposition.
\end{proof}

\begin{thm}
  For any $n \in \N$,  $Beq_o(n)$ and $Beq_h(n)$ admit the unique Haar states.
  We write them $h_o$ and $h_{h}$, respectively.  In particular, we have $h_o \circ \iota_n = h_{I_o}$ and  $h_h \circ \iota_n = h_{I_h}$.
\end{thm}

\begin{proof}
As $||U^x_{ij}||_n \leq 1$, we can extend $\pi_x$ to $Beq_x$ $(x=0,h)$, which proves the theorem.
\end{proof}

\section{Boolean De Finetti theorems}

Let  $(M, \varphi)$ be a pair of a von Neumann algebra and a normal state with faithful GNS-representation and consider an infinite sequence $(x_j)_{j \in \N}$ of self-adjoint elements  $x_j \in M$.
We may assume $M \subset B(H)$,   and $\varphi$ is implemented by $\Omega \in H$, which is a cyclic vector for $M$. %
Throughout this section we suppose $\mathrm{ev}_x(\ms P_{\infty}^o)$ is $\sigma$-weakly dense in $M$, where  $\mathrm{ev}_x$ is the evaluation map
(see Notation \ref{note_intro}. for the definition).

\subsection{Combinatorial part} \hfill

At first we show the purely combinatorial part of Boolean de Finetti theorems.
\begin{prop}\label{easy_part}
	Assume $D$ be a blockwise category of interval partitions.
	Let $E \colon M \ra N$ be a $\varphi$-preserving conditional expectation.
	Suppose $(x_j)_{j \in J}$ are Boolean independent and identically distributed over  $(E, N)$,  and
	$
	K^E_k [x_1, x_1, \dotsc, x_1] = 0,
	$
	for all $k  \in \N \setminus L_D$.
	Then $(x_j)_{j \in \N}$ is  $\mathcal{A}_p[D]$-invariant.
\end{prop}

\begin{proof}
	By the moments-cumulants formula, we have for any $\bfj \in [n]^k$ and $k \in \N$,
	\begin{align*}
	(\varphi \circ \mathrm{ev}_x \otimes \mathrm{id} )\circ \Psi_n (X_{j_1} \dotsb X_{j_k})
	& = \sum_{\bfi \in [n]^k }\varphi(x_{i_1} \dotsb x_{i_k}) \otimes pu_{i_1j_1} \dotsb u_{i_kj_k}p \\
	& = \sum_{\bfi \in [n]^k }\sum_{\substack{\pi \in D(k) \\ \pi \leq \ker\bfi}}%
	K^{(\pi)}_E[x_1, \dotsc, x_1] \otimes pu_{i_1j_1} \dotsb u_{i_kj_k}p \\
	& = \sum_{\pi \in D(k)} K^{(\pi)}_E[x_1, \dotsc, x_1] \otimes %
	\sum_{\substack{\bfi \in [n]^k \\ \pi \leq \ker\bfi}} p u_{i_1j_1} \dotsb u_{i_kj_k} p \\
	& = \sum_{\substack{\pi \in D(k) \\ \pi \leq \ker\bfj}}K^{(\pi)}_E[x_1, \dotsc, x_1] \otimes p \\
	& = \varphi \circ \mathrm{ev}_x (X_{j_1} \dotsb X_{j_k}) \otimes p.
	\end{align*}
\end{proof}

\subsection{Observations on the conditional expectations }\hfill

To prove the opposite direction, we observe properties of the conditional expectations.
Throughout this section, we assume  $D$ is a blockwise category of interval partitions.

\begin{note}\hfill

\begin{enumerate}
	\item Denote by $\ms P_\infty^{o, \Psi_n}$ the fixed point algebra of the coaction $\Psi_n$, that is,
	\begin{align*}
	\ms P_\infty^{o, \Psi_n} := \{ f \in \ms P_{\infty}^o \mid \Psi_n(f) = f \otimes p \}.
	\end{align*}
	\item Define a linear map $E_n \colon \ms P_{\infty}^o \ra \ms P_{\infty}^o$ by  $ E_n :=  (\mathrm{id} \otimes h) \circ \Psi_n$.
	\item For $\pi \in P(k)$, we set
  \begin{align*}
  X_\pi := \sum_{\bfj \in [n]^k, \ \pi \leq \ker\bfj } X_{j_1} \dotsb X_{j_k}.
  \end{align*}

\end{enumerate}
\end{note}

\begin{prop}
The following hold:
  \begin{enumerate}
  \item
  $\Psi_n$ is $\ms P_\infty^{o, \Psi_n}$-$\ms P_\infty^{o, \Psi_n}$ bilinear map : for each $f \in \ms P_\infty^{o, \Psi_n}$ and  $g \in \ms P_\infty^o$,
    \begin{align*}
    \Psi_n(fg) = (f \otimes \mathrm{id}) \Psi_n(g),  \ \Psi_n(gf) = \Psi_n(g)(f \otimes \mathrm{id}).
    \end{align*}
  \item
  $\mathcal{E}_n$ is a conditional expectation with respect to the embedding $\ms P_\infty^{o, \Psi_n} \hra \ms P_{\infty}^o$.
  \end{enumerate}
\end{prop}

\begin{proof}
By \eqref{fix}, it follows that $\ms P_\infty^{o, \Psi_n} = \mathrm{Span}\{ X_\pi \in \ms P_\infty^o \mid \pi \in D(k), k \in \N \}$.
For any $\bfj \in [n]^k, \pi \in D(l)$ and $k, l \in \N$,
\begin{align*}
\Psi_n(X_{i_1} \dotsb X_{i_k} X_\pi  ) = \Psi_n(X_{i_1} \dotsb X_{i_k})(X_\pi \otimes \mathrm{id})
\end{align*}
by the direct computation.
The symmetric proof shows   $\Psi_n$ is a $\ms P_\infty^{o, \Psi_n}$-$\ms P_\infty^{o, \Psi_n}$ bilinear map.

Next, we prove that $\mathcal{E}_n$ is a conditional expectation. $\mathcal{E}_n$ is also $\ms P_\infty^{o, \Psi_n}$-$\ms P_\infty^{o, \Psi_n}$ bilinear map since so is $\Psi_n$. Clearly we have $\mathcal{E}_n[f] = (id \otimes h)(f \otimes p) = f$ for any  $f \in \ms P_\infty^{o, \Psi_n}$.
The proof is completed by showing that $\Psi_n \circ \mathcal{E}_n = \mathcal{E}_n[\cdot] \otimes p$.
Let $\nu$ be the natural isomorphism $V^D_n \otimes \ms \C \ra V^D_n$. Then
  \begin{align*}
  \Psi_n \circ \mathcal{E}_n[f] %
  = (\mathrm{id} \otimes \nu) \circ (\Psi_n \circ \mathrm{id}) \circ (\mathrm{id} \otimes h) \circ \Psi_n
  = (\mathrm{id} \otimes \nu) \circ (\mathrm{id} \otimes \mathrm{id} \otimes h) \circ (\Psi_n \otimes \mathrm{id}) \circ \Psi_n.
  \end{align*}
As $\Psi_n$ is a linear coaction, the right-hand side is equal to
  $
  (\mathrm{id} \otimes \nu) \circ (\mathrm{id} \otimes \mathrm{id} \otimes h) \circ (\mathrm{id} \otimes \Delta) \circ \Psi_n.
  $
By the invariance property of the Haar functional $h$, this is equal to
  $
  (\mathrm{id} \otimes \nu) \circ \iota \circ (\mathrm{id} \otimes h) \circ \Psi_n,
  $
where $\iota$ is the embedding $\ms P_\infty^o \otimes \C \hra \ms P_\infty^o \otimes V^D_n \otimes \C$; $\iota(f \otimes \lambda) = f \otimes p \otimes \lambda$.
By the easy computation, this is equal to $\mathcal{E}_n[\ \cdot \ ] \otimes p$.
\end{proof}

Using the invariance of the joint distribution,  we see that the conditional expectation is connected with the $L^2$-conditional expectation.
\begin{lem}\label{Jones projection}
Suppose $(x_j)_{j \in \N}$ is $\mathcal{A}_p[D]$-invariant for  a blockwise category $D$ of interval partitions, or
 $Beq_x$-invariant for $x=s,o,h$.
  Then  $\mathcal{E}_n$ preserves $\varphi \circ \mathrm{ev}_x$ for any $n \in \N$.
  Moreover for any $f \in \ms P_{\infty}^o$, we have
  \begin{align*}
  e_n \mathrm{ev}_x(f) e_n = \mathrm{ev}_x(\mathcal{E}_n(f)) e_n,
  \end{align*}
  where $e_n$ is the orthogonal projection onto $\overline{\mathrm{ev}_x(\ms P^{\Psi_n}) \Omega}$.
\end{lem}

\begin{proof}
	By definition $\mathcal{A}_p[D]$-invariance implies that   $\mathcal{E}_n$ preserves $\varphi \circ \mathrm{ev}_x$.
	Assume $Beq_x$-invariance.
	Since $h_{I_x} = h_x \circ \iota_n$, we have
	$
	\mathcal{E}_n = \left( \text{id} \otimes \left(h_x \circ \iota_n \right) \right)\Psi_n %
	= (id \otimes h_x) \Phi_n
	$
	.
	 The $Beq_x$-invariance implies that  $\mathcal{E}_n$ preserves $\varphi \circ \mathrm{ev}_x$.
	For any $\pi, \sigma \in D(k)$ and $f \in \ms P_\infty^o$, we have
	\begin{align*}
	\langle  X_\pi \Omega, ev_x \circ \mathcal{E}_n(f) f_\sigma \Omega \rangle %
	= \varphi(\text{ev}_x \circ \mathcal{E}_n(X_\pi^* f X_\sigma) )%
	= \varphi(\text{ev}_x(X_\pi^* f X_\sigma)) %
	=	\langle  X_\pi \Omega, ev_x(f) f_\sigma \Omega \rangle,
	\end{align*}
	which completes the proof.
\end{proof}

In \cite{banica2012finetti}, a noncommutative martingale convergence theorem of cumulants plays an important role in the proof of de Finetti theorems.
Since $\varphi$ is not faithful, we modify this convergence theorem.

 \begin{prop}\label{prop, martin}
 Let  $ (M \subset B(H),\ \Omega \in H ) $ be a pair of a von Neumann algebra and a cyclic vector.
 Assume $M$ is $\sigma$-weekly generated by a sequence $(x_n)_{n \in \N}$ of self-adjoint elements.
 Let $q \in M $ be a non-zero projection and $L := qMq$,
 set a conditional expectaion $E_L := q \cdot q \colon M \ra L$.
 Let $(\ms B_n)_{n\in \N}$  be a decreasing sequence of $*$-subalgebras  of $\ms P_{\infty}^o$, and  denote by $e_n$ the orthogonal projections onto the closed subspaces $\overline{\mathrm{ev}_x(\ms B_n)\Omega}$. Set
 	 \[
 	 B_\infty := \bigcap_{n \in \N} \mathrm{ev}_x(\ms B_n).
 	 \]
 We assume the following conditions:
 	  \begin{enumerate}
 		 \item\label{martin 1-a}
 		 There is a $\varphi \circ \mathrm{ev}_x$ preserving conditional expectation $\mathcal{E}_n \colon \ms P_{\infty}^o \ra \ms B_n$ for each $n \in \N$.
 		 \item\label{martin 1-b}
 		 $\overline{B_\infty \Omega} = \overline{L\Omega}$.
 	  \end{enumerate}
 Then for any $\pi \in I(k)$, $k \in \N$, and $f_1, \dotsc, f_k \in \ms P_{\infty}^o$, we have
 	 \begin{align*}
 	 \text{s-}\underset{n \ra \infty}{\lim}\mathrm{ev}_x(\mathcal{E}_n^{\pi}[f_1, \dotsc, f_k])e_n &= E^{(\pi)}_\mathrm{nut}[f_1(x), \dotsc, f_k(x)], \\
 	 \text{s-}\underset{n \ra \infty}{\lim}\mathrm{ev}_x(K^{\mathcal{E}_n}_\pi[f_1, \dotsc, f_k])e_n &= K^{E_L}_\pi[f_1(x), \dotsc, f_k(x)],
 	 \end{align*}
 	 where we write $f(x) = ev_x(f)$ for $ f \in \ms P_\infty^o$.
\end{prop}

 \begin{proof}
 	By condition (1), $e_n \mathrm{ev}_x(f) e_n$ $= \mathrm{ev}_x(\mathcal{E}_n(f)) e_n$.
 	By condtition (2), s-$\lim_{n \ra \infty}e_n = q$, and s-$\lim_{n \ra \infty} \mathrm{ev}_x(\mathcal{E}_n(f)) e_n = q\mathrm{ev}_x(f)q=
 	E_L[\mathrm{ev}_x(f)].$
 	It holds that
 	$
 	\mathrm{ev}_x \circ \mathcal{E}_n^{\pi}[f_1, \dotsc, f_k] e_n =%
 	$
 	$\prod_{V \in \pi}^{\ra} e_n \mathrm{ev}_x(\prod_{j \in V }^{\ra} f_j) e_n,
 	$
 	for any $\pi \in I(k)$.
 	Hence
 	\[
 	\text{s-}\lim_{n \ra \infty}\mathrm{ev}_x \circ \mathcal{E}_n^{\pi}[f_1, \dotsc, f_k] e_n =%
 	\prod_{V \in \pi}^{\ra}E_L[\prod_{j \in V }^{\ra}f_j(x)]=%
 	E_L^\pi[f_1(x), f_2(x), \dots, f_k(x)].
 	\]

 	Partitioned cumulants are linear combinations of partitioned conditional expectations, which proves the statement.

 \end{proof}

\begin{prop}\label{part_exp}
 For any $k \in \N$, $\pi \in D(k)$ and sufficiently large $n$ such that the Gram matrix is invertible, we have
  \begin{align*}
  \mathcal{E}_n^{\pi}[X_1, \dotsc, X_1] = \frac{1}{n^{|\pi|}}\sum_{\substack{\bfi \in [n]^k \\ \pi \leq \ker\bfi}} X_{i_1}X_{i_2}\dotsb X_{i_k}.
  \end{align*}

\end{prop}
\begin{proof}
	This follows by a similar proof to that in \cite[Prop.4.7]{banica2012finetti}, which is induction on $|\pi|$.
\end{proof}

\begin{lem}\label{lem_eMe}
	Let $M$ be a von Neumann algebra. Fix a nonzero projection $e \in M$.
	Set a conditional expectation $E \colon M \ra N = eMe$ by  $E(y) = eye$.
	Let $k\in \N$ with $k \geq 2$ and $\pi \in I(k)$.
	Assume that $l \in \N$ satisfies $l < k$ and $l \sim^\pi l+1$.
	Then for any  $b \in N, y_1, \dotsc, y_k \in M$,
	\begin{align}\label{eMe}
	K^E_\pi[y_1, \dotsb, y_l b, y_{l+1}, \dotsb, y_k ] = 0.
	\end{align}

\end{lem}

\begin{proof}
	In the case $k=2$, it holds $I(2)= \{ \sqcap \}$ and
	$
	K^E_2[y_1b, y_2] = E[y_1b y_2] - E[y_1b]E[y_2] = ey_1by_2e-ey_1bey_2e=0
	$
	as $b=be$.

	Let $k \geq 3$.
	Assume \eqref{eMe} holds for any $\pi \in I(k-1)$.
	Since $b=be$,
	$
	E[y_1\dotsc y_l b y_{l+1} \dotsc y_k ]%
	$
	$=E[y_1\dotsc y_lb]  E[y_{l+1} \dotsc y_k ].
	$
	The moments-cumulants formula and the assumption of induction imply that
	\begin{align*}
	K^E_k[y_1\dotsb, y_l b&, y_{l+1}, \dotsb, y_k ] \\%
	&= E[y_1\dotsc y_l b y_{l+1} \dotsc y_k ] - %
	\sum_{\pi \in I(k), \pi \neq {\bf1_{k}}}K_\pi^E [y_1, \dotsb, y_l b, y_{l+1}, \dotsb, y_k] \\
	& = E[y_1\dotsc y_lb]  E[y_{l+1} \dotsc y_k ] - %
	\sum_{\pi \in I(k), l \not \sim^\pi l+1}K_\pi^E [y_1, \dotsb, y_l b, y_{l+1}, \dotsb, y_k].
	\end{align*}
	We have   $\{\pi \in I(k) \mid  l \not \sim^\pi l+1\}=\{\sigma \otimes \rho \mid \sigma \in I(l), \rho\in I(k-l)\}$. Then
	\begin{align*}
	K^E_{\sigma \otimes \rho} [y_1, \dotsb, y_lb, y_{l+1}, \dotsb, y_k]
	&= \prod_{V \in \sigma \otimes \rho}^{\ra} %
	K^E_{(V)} [y_1, \dotsb, y_lb, y_{l+1}, \dotsb, y_k]\\
	&=\prod_{V_1 \in \sigma}^{\ra}%
	K^E_{(V_1)} [y_1, \dotsb, y_lb]%
	\prod_{V_2 \in  \rho}^{\ra}%
	K^E_{(V_2)} [y_{l+1}, \dotsb, y_k]	\\
	&=K^E_{\sigma} [y_1, \dotsb, y_lb]%
	K^E_{\rho} [y_{l+1}, \dotsb, y_k].
	\end{align*}
	Hence
	$ E[y_1\dotsc y_lb]  E[y_{l+1} \dotsc y_k ] - %
	\sum_{\pi \in I(k), l \not \sim^\pi l+1}K_\pi^E [y_1, \dotsb, y_l b, y_{l+1}, \dotsb, y_k]=0.
	$
	Induction on $k$ proves the lemma.
\end{proof}

\subsection{Boolean de Finetti theorems} \hfill

\begin{lem}\label{lemma i.d. cumulants}
	Assume that $||x_j|| \leq ||x_1||$ for any $ j \in \N$.
Let $D$ be one of  $I, I_o, I_h, I_b$.
For any $k \in \N$, $\sigma \in D(k)$ and $n_0, n \in \N$ with $n_0 \leq n$, set an element in ${\ms P}^o_{\geq n_0}$
by
\begin{align*}
f^{n_0,n}_\sigma := \sum_{\pi \in D(k)}%
\frac{1}{n^{|\pi|}}%
\sum_{\substack{\bfi \in [n_0, n]^k \\ \pi \leq \ker\bfi }}%
X_{i_1}X_{i_2}\dotsb X_{i_k}\mu_{I(k)}(\pi, \sigma).
\end{align*}
Then we have
	\begin{align}
	&|| \mathrm{ev}_x \circ \mathcal{E}_n [X_{j_1}X_{j_2} \dotsb X_{j_k}]-%
	\sum_{\substack{\sigma \in D(k) \\ \sigma \leq \ker\bfj }}%
	\mathrm{ev}_x \circ K^{\mathcal{E}_n}_\sigma[X_1, \dotsc, X_1]||%
		\ra 0 \ (\text{ as } n \ra \infty). \label{i.d. cumulants}\\
	&||\mathrm{ev}_x \circ \mathcal{E}_n[X_{j_1}X_{j_2} \dotsb X_{j_k}]-%
	\sum_{\substack{\sigma \in D(k) \\ \sigma \leq \ker\bfj }}%
	\mathrm{ev}_x(f^{n_0,n}_\sigma) ||
	\ra 0 \ (\text{as } n \ra \infty). \label{i.d. cumulants faraway}
	\end{align}

\end{lem}

\begin{proof}

	By Proposition \ref{rmk_wein} and Lemma \ref{part_exp}, we have for sufficiently large $n$,
	  \begin{align*}
	   \mathcal{E}_n[X_{j_1}X_{j_2} \dotsb X_{j_k}]  %
	  &= \sum_{\bfi \in [n]^k} X_{i_1}X_{i_2}\dotsb X_{i_k} Q_{\bfi \bfj}^{(k)} \\
	  &= \sum_{\bfi \in [n]^k} X_{i_1}X_{i_2}\dotsb X_{i_k} \sum_{\substack{\pi, \sigma \in D(k) \\ \pi \leq \ker\bfi, \sigma \leq \ker\bfj}} W_{k,n}(\pi, \sigma) \\
	  &= \sum_{\substack{\sigma \in D(k) \\ \sigma \leq \ker\bfj }} %
	  \sum_{\pi \in D(k)}  (\frac{1}{n^{|\pi|}}\sum_{\substack{\bfi \in [n]^k\\ \pi \leq \ker\bfi}} X_{i_1}X_{i_2}\dotsb X_{i_k}) %
	  n^{|\pi|} W_{k,n}(\pi, \sigma) \\
	  &=\sum_{\substack{\sigma \in D(k) \\ \sigma \leq \ker\bfj }} %
	  \sum_{\pi \in D(k)} \mathcal{E}_n^{\pi}[X_1, \dotsc, X_1] %
	  n^{|\pi|} W_{k,n}(\pi, \sigma).
	  \end{align*}
	By the moments-cumulants formula Proposition~\ref{moments-cumulants formula}, we have
	  \begin{align*}
	  \mathcal{E}_n&[X_{j_1}X_{j_2} \dotsb X_{j_k}]-%
	  \sum_{\substack{\sigma \in D(k) \\ \sigma \leq \ker\bfj }}%
	  K^{\mathcal{E}_n}_\sigma[X_1, \dotsc, X_1]  \\
	  &=\sum_{\substack{\sigma \in D(k) \\ \sigma \leq \ker\bfj }} %
	  \sum_{\pi \in D(k)} \mathcal{E}_n^{\pi}[X_1, \dotsc, X_1] %
	  n^{|\pi|} W_{k,n}(\pi, \sigma)-%
	  \sum_{\substack{\sigma \in D(k) \\ \sigma \leq \ker\bfj }}%
	  \sum_{\pi \in D(k)}%
	  \mathcal{E}_n^{\pi}[X_1, \dotsc, X_1]\mu_{I(k)}(\pi, \sigma) \\
	  &=\sum_{\pi \in D(k)} %
	  [\sum_{\substack{\sigma \in D(k) \\ \sigma \leq \ker\bfj }} %
	  n^{|\pi|} W_{k,n}(\pi, \sigma) - \mu_{I(k)}(\pi, \sigma)]%
	  \mathcal{E}_n^{\pi}[X_1, \dotsc, X_1].
	  \end{align*}
	  \begin{align*}
		||\mathrm{ev}_x \circ \mathcal{E}_n&[X_{j_1}X_{j_2} \dotsb X_{j_k}]-%
		\sum_{\substack{\sigma \in D(k) \\ \sigma \leq \ker\bfj }}%
		\mathrm{ev}_x \circ K^{\mathcal{E}_n}_\sigma[X_1, \dotsc, X_1]|| \\
	  &\leq \max_{\pi \in D(k)} [\sum_{\substack{\sigma \in D(k) \\ \sigma \leq \ker\bfj }} %
		|n^{|\pi|} W_{k,n}(\pi, \sigma) - \mu_{I(k)}(\pi, \sigma)| ]%
		 \sum_{\pi \in D(k)} %
		 ||\mathrm{ex}_x \circ \mathcal{E}_n^{\pi}[X_1, \dotsc, X_1] || \\
		 & \leq \max_{\pi \in D(k)} \sum_{\substack{\sigma \in D(k) \\ \sigma \leq \ker\bfj }} %
		 |n^{|\pi|} W_{k,n}(\pi, \sigma) - \mu_{I(k)}(\pi, \sigma)| %
		 \cdot |D(k)| \cdot ||x_1||^k.
	\end{align*}
	By the Weingarten estimate in Proposition \ref{prop_wein_est},
	\begin{align*}
	\max_{\pi \in D(k)} \sum_{\substack{\sigma \in D(k) \\ \sigma \leq \ker\bfj }} %
	|n^{|\pi|} W_{k,n}(\pi, \sigma) - \mu_{I(k)}(\pi, \sigma)| %
	= O(\frac{1}{n})  \ (\text{as } n \ra \infty).
	\end{align*}
	Therefore, we have \eqref{i.d. cumulants}.

	For any $n_0 \in \N$, we have
	\begin{align*}
		K^{\mathcal{E}_n}_\sigma[X_1, \dotsc, X_1] - f^{n_0,n}_\sigma
		=\sum_{\pi \in D(k)}%
		\frac{1}{n^{|\pi|}}%
		\sum_{\substack{\bfi \in [n]^k \setminus [n_0, n]^k \\ \pi \leq \ker\bfi }}%
		 X_{i_1}X_{i_2}\dotsb X_{i_k}\mu_{I(k)}(\pi, \sigma).
	\end{align*}
	Now
	\begin{align*}
	&\frac{1}{n^{|\pi|}}%
	\sum_{\substack{\bfi \in [n]^k \setminus [n_0, n]^k \\ \pi \leq \ker\bfi }}%
	||x_{i_1}x_{i_2}\dotsb x_{i_k} ||
	\leq \frac{n^{|\pi|} - (n-n_0)^{|\pi|}}{n^{|\pi|}} ||x_1||^k
	 \ra 0 \ (\text{as } n \ra \infty).
	\end{align*}
	Hence
	\begin{align*}
	||\mathrm{ev}_x \circ \mathcal{E}_n&[X_{j_1}X_{j_2} \dotsb X_{j_k}]-%
	\sum_{\substack{\sigma \in D(k) \\ \sigma \leq \ker\bfj }}%
	\mathrm{ev}_x(f^{n_0,n}_\sigma) || \\
	&\leq ||\mathrm{ev}_x \circ \mathcal{E}_n[X_{j_1}X_{j_2} \dotsb X_{j_k}]-%
	\sum_{\substack{\sigma \in D(k) \\ \sigma \leq \ker\bfj }} %
	\mathrm{ev}_x \circ K^{\mathcal{E}_n}_\sigma[X_1, \dotsc, X_1]|| 	\\
	&\hspace{30mm} + \sum_{\substack{\sigma \in D(k) \\ \sigma \leq \ker\bfj }}%
	||\mathrm{ev}_x \circ  K^{\mathcal{E}_n}_\sigma[X_1, \dotsc, X_1] - \mathrm{ev}_x (f^{n_0,n}_\sigma)|| \\
	&\ra 0 \ (\text{as } n \ra \infty).
					\end{align*}

\end{proof}

Now we are prepared to prove our main theorem, de Finetti theorems for  ${\mathcal A}_p[I_x]$ and $Beq_x$.

\begin{thm}\label{finetti}
Let $(M, \varphi)$ be a pair of a von Neumann algebra and a nondegenerate normal state.
Assume $M$ is generated by self-adjoint elements $(x_j)_{j \in \N}$.
Consider the following three assertions.
	\begin{enumerate}
 	\item \label{de finetti 1}
 	The joint distribution of $(x_j)_{j \in \N}$ is $\mathcal{A}_p[I_x]$-invariant.
	\item \label{de finetti 2}
	The joint distribution of $(x_j)_{j \in \N}$ is $Beq_x$-invariant.

 	\item \label{de finetti 3}
 	The elements $(x_j)_{j \in \N}$ are Boolean independent and identically distributed over $(E_\mathrm{nut}, M_\mathrm{nut})$,
 	and for all $k \in \N \setminus L_{I_x}$, and $b_1, \dotsb, b_k \in M_\mathrm{nut} \cup \{1 \}$, it holds that
 	\begin{align*}
 	K_k^{E_\mathrm{nut}}[x_1 b_1, x_1 b_2, \dotsc, x_1] = 0.
 	\end{align*}

 	\end{enumerate}
Then for $x=s,o,h$, all assertions are  equivalent.
For $x=b$, (1) and (3) are equivalent.
\end{thm}
\begin{proof}
By Proposition \ref{easy_part}, we have \eqref{de finetti 3} implies \eqref{de finetti 1}.
We prove each condition \eqref{de finetti 1}, \eqref{de finetti 2} implies \eqref{de finetti 3} in the case $x=s, o,h$, and prove \eqref{de finetti 1} implies \eqref{de finetti 3} in the case $x=b$.
Let $(H,\Omega)$ be the GNS-representation of $(M, \varphi)$.
As $\varphi$ is nondegenerate, we may assume $M \subset B(H)$.
Set $B_\infty := \bigcap_{n \in \N} \mathrm{ev}_x(\ms P^{\Psi_n})$.
At first, we prove $\overline{B_\infty \Omega} = \overline{M_\mathrm{nut} \Omega}$.
Since $\ms P^o_{\geq n} \subset \ms P^{\Psi_n}$,
it is clear that  $ \overline{B_\infty \Omega} \supset \overline{M_\mathrm{nut} \Omega}$.
Let $e_n$ be the orthogonal projection onto the subspace $H_n := \overline{\text{ev}_x(\ms P^{\Psi_n})} \Omega \subset H$.
Set $e_\infty$ be the orthogonal projection onto $\cap_{n \in \infty} H_n = \overline{B_\infty \Omega}$. The projections $(e_n)_{n \in \N}$ strongly converges to $e_\infty$.
To see $ \overline{B_\infty \Omega} \subset \overline{M_\mathrm{nut} \Omega}$, we only need to show that
$
e_\infty x_\bfj \Omega \in \overline{M_\mathrm{nut}\Omega} $
 for any  $k \in \N, \bfj \in [n]^k$.
By Lemma~\ref{Jones projection}, each condition \eqref{de finetti 1} , \eqref{de finetti 2} implies
$
  \mathrm{ev}_x \circ \mathcal{E}_n[X_{j_1}X_{j_2} \dotsb X_{j_k}] \Omega %
  =e_n  x_{j_1}x_{j_2} \dotsb x_{j_k} \Omega
$
.
As each condition \eqref{de finetti 1}, \eqref{de finetti 2} implies  that $(x_j)_{j \in \N}$ are identically distributed, we have $||x_j||=||x_1||$ for any $j \in \N$.
Then by Lemma~\ref{lemma i.d. cumulants}, it holds that
$\text{ev}_x(f^{n_0,n}_\sigma) \Omega $ converges to an element in $\overline{\mathrm{ev}_x ({\ms P}^o_{\geq n_0} )\Omega}$ as $n \ra \infty$.
We have
\begin{align*}
e_\infty x_\bfj \Omega = \lim_{n \ra \infty} \mathrm{ev}_x \circ \mathcal{E}_n[X_{j_1}X_{j_2} \dotsb X_{j_k}] \Omega %
  \in \bigcap_{n_0 \in \N}\overline{\text{ev}_x({\ms P}^o_{\geq n_0}) \Omega}%
  = \overline{M_\mathrm{nut}\Omega}.
\end{align*}

By Lemma~\ref{Jones projection}, $\mathcal{E}_n$ preserves $\varphi \circ \text{ev}_x$ and
by the modified martingale convergence theorem (see Proposition~\ref{prop, martin}) and \eqref{i.d. cumulants}, we obtain   for any $j_1, \dotsc, j_k \in J, k \in \N$,

  \begin{align}\label{main_step1}
  E_\mathrm{nut}[x_{j_1}\dotsb x_{j_k}] = \sum_{\substack{\sigma \in D(k) \\ \sigma \leq \ker\bfj }}K^{E_\mathrm{nut}}_{\sigma}[x_1, \dotsc, x_1].
  \end{align}

The proof is completed by showing that %
for any $b_0, \dotsc, b_k \in M_\mathrm{nut} \cup \{1\},$ $j_1, \dotsc, j_k \in J$, and $k \in \N$,
  \begin{align}\label{main_step2}
  E_\mathrm{nut}[x_{j_1} b_1x_{j_2} b_2 \dotsb b_{k-1}x_{j_k}] = %
  \sum_{\substack{\sigma \in D(k) \\ \sigma \leq \ker\bfj }}K^{E_\mathrm{nut}}_{\sigma}[x_1 b_1, x_1 b_2, \dotsc, x_1].
  \end{align}
We prove this by  induction on $ \# \{ l \in [k-1] ; b_l \neq 1 \}$.
In the case $ \# \{ l \in [k-1] ; b_l \neq 1 \}$ = 1, the claim holds by \eqref{main_step1}.
Pick any $m \in \N \cup \{0\}$ with $m \leq k-1$.
Assume that \eqref{main_step2} is proved in the case that  $ \# \{ l \in [k-1] ; b_l \neq 1 \} < m$.
Consider the case $\# \{ l \in [k-1] ; b_l \neq 1 \} = m$.
Let $r = \max \{ l \in [k-1] ; b_l \neq 1 \}$.
Then by Lemma \ref{lem_eMe},

\begin{align*}
\sum_{\substack{\sigma \in D(k)\\ \sigma \leq \ker\bfj}}%
	  K^{E_\mathrm{nut}}_{\sigma}[x_1 b_1, \dotsc, x_1 b_{r}, \dotsc, x_1]= %
	   \sum_{\substack{\sigma \in D(k), \sigma \leq \ker\bfj \\ r \underset{\sigma}{\not\sim} r+1}}%
	  K^{E_\mathrm{nut}}_{\sigma |_{[1,r]}}[x_1 b_1, \dotsc, x_1] b_{r} %
	  K^{E_\mathrm{nut}}_{\sigma |_{[r+1, k]}}[x_1 b_{r+1}, \dotsc, x_1].
	  \end{align*}
By the property (D1), this equals to
  \begin{align*}
  \sum_{\substack{\pi \in D(r) \\ \pi \leq \ker\bfj|_{[1,r]} }} %
  & K^{E_\mathrm{nut}}_{\pi}[x_1 b_1, \dotsc, x_1]b_{r}%
  \sum_{\substack{\rho \in D(k-r) \\ \rho \leq \ker\bfj|_{[r+1, k]} }}K^{E_\mathrm{nut}}_\rho[x_1 b_{r+1}, \dotsc, x_1]\\
  &=E_\mathrm{nut}[x_{j_1} b_1 \dotsb x_{j_{r}}] b_{r} E_\mathrm{nut}[x_{j_{r+1}} b_{r+1} \dotsb x_{j_k}] %
  =E_\mathrm{nut}[x_{j_1} b_1x_{j_2} b_2 \dotsb b_{k-1}x_{j_k}].
  \end{align*}
By induction on $m$,  \eqref{main_step2} holds for any $b_0, \dotsc, b_k \in M_\mathrm{nut} \cup \{1\}$, which proves \eqref{de finetti 1}.
\end{proof}

\begin{cor}\label{finetti_hbo}
 If the equaivalent conditions in Theorem \ref{finetti} are satisfied for one of $x= o, h$ and $b$, then the following hold:
 \begin{enumerate}
\item[(o)] If $x=o$, $(x_j)_{j \in \N}$ form a  $M_\mathrm{nut}$-valued Boolean centered Bernoulli family.
\item[(h)] If $x=h$, $(x_j)_{j \in \N}$ are Boolean independent, and have even and identically distributions, over $M_\mathrm{nut}$.
\item[(b)] If $x=b$,  $(x_j)_{j \in \N}$ form a  $M_\mathrm{nut}$-valued Boolean shifted Bernoulli family.

\end{enumerate}

\end{cor}
\begin{proof}
The proof directly follows from Theorem \ref{finetti}.
\end{proof}

\section*{Acknowledgements}
We would like to thank Yasuyuki Kawahigashi for continuing support.
We would like to express our gratitude to Moritz Weber for incisive comments and valuable advice.
We feel deep graduate for Takahiro Hasebe for useful discussion and comments for the first version of our preprint.
Discussions with Weihua Liu  have been illuminating. He pointed out that the generators of our Boolean quantum semigroups in the original version are  possibly unbounded.

\bibliographystyle{plain}
\bibliography{reference_dbeq}

\end{document}